\documentclass[a4paper,10pt]{amsart}
\usepackage[utf8]{inputenc}
\usepackage{lmodern}
\usepackage[T1]{fontenc}
\usepackage{microtype} 	
\usepackage{amsmath,amssymb,amsfonts,amsthm,latexsym,mathrsfs}
\usepackage{color}
\usepackage{hyperref}
\hypersetup{colorlinks=true,linkcolor=red,citecolor=blue}
\usepackage{multirow}
\usepackage{booktabs}
\usepackage{longtable}
\theoremstyle{plain}
\newtheorem{theorem}[subsection]{{\bf Theorem}}
\newtheorem*{theorem*}{{\bf Theorem}}
\newtheorem{corollary}[subsection]{{\bf Corollary}}
\newtheorem*{corollary*}{{\bf Corollary}}
\newtheorem{proposition}[subsection]{{\bf Proposition}}
\newtheorem{lemma}[subsection]{{\bf Lemma}}

\theoremstyle{definition}

\theoremstyle{remark}

\newtheorem{example}[subsection]{{\it Example}}
\numberwithin{equation}{section}

\DeclareMathOperator{\HH}{H}

\DeclareMathOperator{\B}{B}
\DeclareMathOperator{\M}{M}

\DeclareMathOperator{\cp}{cp}
\DeclareMathOperator{\kk}{k}

\DeclareMathOperator{\Out}{Out}

\newcommand{\QZ}{\mathbb{Q}/\mathbb{Z}}
\newcommand{\Z}{\mathbb{Z}}
\begin{document}
\baselineskip=14pt
\title[Bogomolov multipliers and commuting probability]
{Universal commutator relations, Bogomolov multipliers, and commuting 
probability}
\author[U. Jezernik]{Urban Jezernik}
\address[Urban Jezernik]{
Institute of Mathematics, Physics, and Mechanics \\
Jadranska 19 \\
1000 Ljubljana \\
Slovenia}
\thanks{}
\email{urban.jezernik@imfm.si}
\author[P. Moravec]{Primo\v{z} Moravec}
\address[Primo\v{z} Moravec]{
Department of Mathematics \\
University of Ljubljana \\
Jadranska 21 \\
1000 Ljubljana \\
Slovenia}
\thanks{}
\email{primoz.moravec@fmf.uni-lj.si}
\thanks{}
\date{\today}
\begin{abstract}
\noindent
Let $G$ be a finite $p$-group. We prove that whenever the commuting probability of $G$ is greater than $(2p^2 + p - 2)/p^5$, the unramified Brauer group of the field of $G$-invariant functions is trivial. Equivalently, all relations between commutators in $G$ are consequences of some universal ones. The bound is best possible, and gives a global lower bound of $1/4$ for all finite groups. The result is attained by describing the structure of groups whose Bogomolov multipliers are nontrivial, and Bogomolov multipliers of all of their proper subgroups and quotients are trivial. Applications include a classification of $p$-groups of minimal order that have nontrivial Bogomolov multipliers and are of nilpotency class $2$, a nonprobabilistic criterion for the vanishing of the Bogomolov multiplier, 
and establishing a sequence of Bogomolov's absolute $\gamma$-minimal factors which are $2$-groups of arbitrarily large nilpotency class, thus providing counterexamples to some of Bogomolov's claims.
In relation to this, we fill a gap in the proof of triviality of Bogomolov multipliers of finite simple groups.
\end{abstract}
\maketitle
\section{Introduction}
\label{s:intro}

\noindent
Let $G$ be a group and $G\wedge G$ the group generated by the symbols $x\wedge y$, where $x,y\in G$, subject to the following relations:
\[	
\begin{aligned}
\label{eq:ext}
xy\wedge z = (x^y\wedge z^y)(y\wedge z), \quad 
x\wedge yz = (x\wedge z)(x^z\wedge y^z), \quad
x\wedge x = 1,
\end{aligned}
\]
where $x,y,z\in G$. The group $G\wedge G$ is said to be the {\it nonabelian exterior square} of $G$. There is a surjective homomorphism $G\wedge G\to [G,G]$ defined by $x\wedge y\mapsto [x,y]$. Miller \cite{Mil52} showed that the kernel $\M (G)$ of this map is naturally isomorphic to the Schur multiplier $\HH_2(G,\mathbb{Z})$ of $G$. In particular, this implies that the relations \eqref{eq:ext}  induce universal commutator identities that hold in a free group, and that $\M(G)$ is, in a sense, a measure of how the commutator identities in $G$ fail to follow from the universal ones only.

Denote $\M _0(G)=\langle x\wedge y\mid x,y\in G,\, [x,y]=1\rangle$ and $\B_0(G)=\M (G)/\M _0(G)$. On one hand, $\B_0(G)$ is an obstruction for the commutator identities of $G$ to follow from the universal commutator identities induced by \eqref{eq:ext} while considering the symbols that generate $\M_0(G)$ as redundant. In the group-theoretical framework, $\B_0(G)$ is accordingly the group of {\it nonuniversal commutator relations} rather than identities that hold in $G$. On the other hand, it is shown in \cite{Mor11} that if $G$ is a finite group and $V$ a faithful representation of $G$ over $\mathbb{C}$, then the dual of $\B_0(G)$ is naturally isomorphic to the unramified Brauer group $\HH^2_{\rm nr}(\mathbb{C}(V)^G,\QZ )$ introduced by Artin and Mumford \cite{Art72}. This invariant represents an obstruction to {\it Noether's problem} \cite{Noe16} asking as to whether the field of $G$-invariant functions $\mathbb{C}(V)^G$ is purely transcendental over $\mathbb{C}$. The crucial part of the proof of the above mentioned result of \cite{Mor11} is based on the ground-breaking work of Bogomolov \cite{Bog88} who showed that $\HH^2_{\rm nr}(\mathbb{C}(V)^G,\QZ )$ is naturally isomorphic to the intersection of the kernels of restriction maps $\HH^2(G,\QZ)\to \HH^2(A,\QZ)$, where $A$ runs through all (two-generator) abelian subgroups of $G$. The latter group is also known as the {\it Bogomolov multiplier} of $G$, cf. \cite{Kun08}. Throughout this paper, we use this terminology for $\B_0(G)$, thus considering the Bogomolov multiplier of $G$ as the kernel of the commutator map from $G \curlywedge G=(G\wedge G)/\M_0(G)$ to $[G,G]$.
This description of $\B_0$ is combinatorial and enables efficient explicit calculations, see \cite{Che13,ARXIVREF,Mor11,Mor12p5} for further details.
In addition to that, it provides an easy proof \cite{Mor14} of the fact that Bogomolov multipliers are invariant with respect to isoclinism, a notion defined by P. Hall in his seminal paper \cite{Hal40} on classifying finite $p$-groups.

In this paper, we consider the problem of triviality of $\B_0$ from the probabilistic point of view. As noted above, the structure of Bogomolov multipliers heavily depends on the structure of commuting pairs of elements of a given group. Denote $\mathcal{S}_G=\{ (x,y)\in G\times G\mid xy=yx\}$. Then the quotient $\cp (G)=|\mathcal{S}_G| / |G|^2$ is the probability that a randomly chosen pair of elements of $G$ commute. Erd\"{o}s and Tur\'{a}n \cite{Erd68} noted that $\cp (G)=\kk (G)/|G|$, where $\kk (G)$ is the number of conjugacy classes of $G$. Gustafson \cite{Gus73} proved an amusing result that if $\cp (G)>5/8$, then $G$ is abelian, and hence $\B_0(G)$ is trivial. We prove the following:

\begin{theorem}
\label{t:cp-p}
Let $G$ be a finite $p$-group. If $\cp (G)>(2p^2 + p - 2)/p^5$, then $\B_0(G)$ is trivial.
\end{theorem}

Homological arguments then give a global bound applicable to all finite groups.

\begin{corollary}
\label{c:cp}
Let $G$ be a finite group. If $\cp (G)>1/4$, then $\B_0(G)$ is trivial.
\end{corollary}

The stated bounds are all sharp. Namely, there exists a group $G$ of order $p^7$ such that $\cp (G)=(2p^2 + p - 2)/p^5$ and $\B_0(G)$ is not trivial. We also mention here that the converse of Theorem \ref{t:cp-p} and Corollary \ref{c:cp} does not hold, nor is there an upper bound on commuting probability ensuring nontriviality of the Bogomolov multiplier.
Examples, as well as some information on groups satisfying the condition of Theorem \ref{t:cp-p}, are provided in the following sections.

The proof of Theorem \ref{t:cp-p} essentially boils down to studying a minimal counterexample. It suffices to consider only finite groups $G$ with $\B_0(G)$ nontrivial, and for every proper subgroup $H$ of $G$ and every proper normal subgroup $N$ of $G$, we have $\B_0(H) = \B_0(G/N) = 0$. We call such groups {\em $\B_0$-minimal groups}. These are a special type of absolute $\gamma$-minimal factors which were introduced by Bogomolov \cite{Bog88}. In the context of nonuniversal commutator relations, $\B_0$-minimal groups are precisely the minimal groups possessing such relations, and may in this way be thought of as the building blocks of groups with nontrivial Bogomolov multipliers. The $\B_0$-minimal groups are thus of independent interest, and the first part of the paper is devoted to describing their structure. A part of this has already been investigated by Bogomolov \cite{Bog88} using cohomological methods; the alternative approach we take via the exterior square provides new proofs and refines that work. Standard arguments show that $\B_0$-minimal groups are $p$-groups. We prove in Theorem \ref{t:frattini} that they can be generated by at most four generators, have an abelian Frattini subgroup, and that their Bogomolov multipliers are of prime exponent, see Corollary \ref{c:exponentB0}. In addition to that, we further explore the structure of $\B_0$-minimal 2-groups, cf. Proposition \ref{p:minimalbreadth}. 

We also consider $\B_0$-minimal groups with respect to isoclinism. We call an isoclinism family to be a {\it $\B_0$-minimal family} if it contains a $\B_0$-minimal group. We observe that all stem groups of a $\B_0$-minimal family are $\B_0$-minimal. The main result in this direction is Theorem \ref{t:class2} which provides a classification of all $\B_0$-minimal isoclinism families of class 2. It turns out that these families are always determined by two stem groups whose presentations can be explicitly written down. Relying on some recent results \cite{Hos11,Hos12,Che13,ARXIVREF} on Bogomolov multipliers of $p$-groups of small orders, Theorem \ref{t:class2} in fact provides a classification of all $p$-groups of order $p^7$ and nilpotency class $2$ with nontrivial Bogomolov multipliers.

In his work on absolute $\gamma$-minimal factors, Bogomolov claimed \cite[Theorem 4.6]{Bog88} that if a finite $p$-group is an absolute $\gamma$-minimal factor, then it is nilpotent of class at most $p$. We use the structural results on $\B_0$-minimal groups we develop in Section \ref{s:minimal}, and also Corollary \ref{c:cp}, to construct a sequence of $\B_0$-minimal $2$-groups with strictly growing nilpotency classes, cf. Example \ref{e:bignilpotency}. This example contradicts the above mentioned Bogomolov's result.
The existence of groups in Example \ref{e:bignilpotency} also contradicts \cite[Lemma 5.4]{Bog88}, which has been used subsequently in proving that the linear and orthogonal finite simple groups have trivial Bogomolov multipliers \cite{Kun08}. We patch up the argument using Corollary \ref{c:cp} in the final section.

\section{$\B_0$-minimal groups}
\label{s:minimal}

\noindent
A finite group $G$ is termed to be a {\em $\B_0$-minimal group} whenever $\B_0(G) \neq 0$ and for every proper subgroup $H$ of $G$ and every proper normal subgroup $N$ of $G$, we have $\B_0(H) = \B_0(G/N) = 0$. The class of $\B_0$-minimal groups is a subclass of the class of absolute $\gamma$-minimal factors defined by Bogomolov \cite{Bog88}. 

Recall the notion of isoclinism of groups introduced by P. Hall \cite{Hal40}. Two groups $G$ and $H$ are {\it isoclinic} if there exists a pair of isomorphisms $\alpha \colon G/Z(G)\to H/Z(H)$ and $\beta \colon [G,G]\to [H,H]$ with the property that whenever $\alpha(a_1Z(G))=a_2Z(H)$ and $\alpha (b_1Z(G))=b_2Z(H)$, then $\beta ([a_1,b_1])=[a_2,b_2]$ for $a_1,b_1\in G$. Isoclinism is an equivalence relation, denoted by the symbol $\simeq$, and the equivalence classes are called {\it families}. Hall proved that each family contains {\it stem groups}, that is, groups $G$ satisfying $Z(G)\le [G,G]$. Stem groups in a given family have the same order, which is the minimal order of all groups in the family. When the stem groups are of order $p^r$ for some $r$, we call $r$ the {\em rank} of the family.

\begin{example}
\label{e:64}
Let $G$ be the group
\[
\left\langle \begin{array}{l|l}
			 \multirow{2}{*}{$a,b,c$} & a^2 = b^2 = 1, \, c^2 = [a,c], \\
			 & [c,b] = [c, a, a],  \, [b,a] \text{ central}, \, \text{class $3$}
		\end{array} \right\rangle.
\]
Another way of presenting $G$ is by a polycyclic generating sequence $g_i$ with $1 \leq i \leq 6$, subject to the following relations:
$g_1^2 = g_2^2 = 1$,
$g_3^2 = g_4 g_5$,
$g_4^2 = g_5$,
$g_5^2 = g_6^2 = 1$,
$[g_2, g_1] = g_6$,
$[g_3, g_1] = g_4$,
$[g_3, g_2] = g_5$,
$[g_4, g_1] = g_5$, and
$[g_i,g_j] = 1$ for other $i>j$.
This is one of the stem groups of the family $\Gamma_{16}$ of \cite{Hal64}. It follows from \cite{Chu09} that $\B_0(G)\cong \mathbb{Z}/2\mathbb{Z}$. Application of the algorithm developed in \cite{Mor11} shows that $\B_0(G)$ is generated by the element $(g_3 \curlywedge g_2)(g_4 \curlywedge g_1)$ in $G \curlywedge G$. The group $G$ is one of the groups of the smallest order that have a nontrivial Bogomolov multiplier \cite{Chu08,Chu09}, so it is also of minimal order amongst all $\B_0$-minimal groups.
\end{example}

The notion of isoclinism is closely related to Bogomolov multipliers. It is shown in \cite{Mor14} that whenever $G$ and $H$ are isoclinic groups, we have $\B_0(G)\cong \B_0(H)$, i.e. the Bogomolov multiplier is a family invariant. A family that contains at least one $\B_0$-minimal group is correspondingly called a {\em $\B_0$-minimal family}. Note that not every group in a $\B_0$-minimal family is itself $\B_0$-minimal. For example, one may take a $\B_0$-minimal group $G$, a nontrivial abelian group $A$, and form their direct product $G \times A \simeq G$. This group is clearly not $\B_0$ minimal. We show, however, that the stem groups of $\B_0$-minimal families are themselves $\B_0$-minimal.

\begin{proposition}
\label{p:families}
In a $\B_0$-minimal family, every group possesses a $\B_0$-minimal section. In particular, the stem groups in the family are all $\B_0$-minimal.
\end{proposition}
\begin{proof}
Let $G$ be a $\B_0$-minimal member of the given isoclinism family and $H \simeq G$ a group that is not a $\B_0$-minimal group. Since $\B_0(H) \cong \B_0(G) \neq 0$, the group $H$ has either a subgroup or a quotient, say $K$, with a nontrivial Bogomolov multiplier. By \cite{Hal40}, the subgroups and quotients of $H$ belong to the same isoclinism families as the subgroups and quotients of $G$. It follows from $\B_0$-minimality of $G$ that the group $K$ must be isoclinic to $H$. As $|K| < |H|$, repeating the process with $K$ instead of $H$ yields a section $S$ of $H$ that is $\B_0$-minimal and isoclinic to $H$. In particular, the stem groups in a $\B_0$-minimal family must be $\B_0$-minimal, since they are groups of minimal order in the family.
\end{proof}

Note also that not all $\B_0$-minimal groups in a given family need be stem, as the following example shows.

\begin{example}
\label{e:256}
Let $G$ be the group generated by elements $g_i$ with $1 \leq i \leq 8$, subject to the following relations:
$g_1^2 = g_5$,
$g_2^2 = g_3^2 = 1$,
$g_4^2 = g_6$,
$g_5^2 = g_7$,
$g_6^2 = 1$,
$g_7^2 = g_8$,
$g_8^2 = 1$,
$[g_2, g_1] = g_4$,
$[g_3, g_1] = g_8$,
$[g_3, g_2] = g_6  g_8$,
$[g_4, g_1] = g_6$,
$[g_4, g_2] = g_6$, and 
$[g_i,g_j] = 1$ for other $i>j$.
Using the algorithm developed in \cite{Mor11}, we see that $G$ is a $\B_0$-minimal group. Its Bogomolov multiplier is generated by the element $(g_3 \curlywedge g_2)(g_4 \curlywedge g_2)(g_3 \curlywedge g_1)$ of order $2$ in $G \curlywedge G$. Since $g_7$ belongs to the center $Z(G)$ but not to the derived subgroup $[G,G]$, the group $G$ is not a stem group. In fact, $G$ is isoclinic to the group given in Example \ref{e:64}, both the isoclinism isomorphisms stemming from interchanging the generators $g_2$ and $g_3$.
\end{example}

Applying standard homological arguments, we quickly observe that $\B_0$-minimal groups
are $p$-groups:
\begin{proposition}
\label{p:pgroups}
A $\B_0$-minimal group is a $p$-group.
\end{proposition}
\begin{proof}
Let $G$ be a $\B_0$-minimal group. Suppose $p$ is a prime dividing the order of $G$. By \cite[Lemma 2.6]{Bog04}, the $p$-part of $\B_0(G)$ embeds into $\B_0(S)$, where $S$ is a Sylow $p$-subgroup of $G$. It thus follows from $\B_0$-minimality that $G$ is a $p$-group.
\end{proof}

Hence $\B_0$-minimal families are determined by their stem $p$-groups. Making use of recent results on Bogomolov multipliers of $p$-groups of small orders \cite{Hos11,Hos12,Chu09,Chu08,Che13,ARXIVREF}, we determine the $\B_0$-minimal families of rank at most $6$ for odd primes $p$, and those of rank at most $7$ for $p = 2$. In stating the proposition, the classifications \cite{Jam80,Jam90} are used.

\begin{proposition}
\label{p:smallb0mins}
The $\B_0$-minimal isoclinism families of $p$-groups with $p$ an odd prime and of rank at most $6$ are precisely the families $\Phi_i$ with $i \in \{ 10, 18, 20, 21, 36 \}$ of \cite{Jam80}.
The $\B_0$-minimal isoclinism families of $2$-groups of rank at most $7$ are precisely the families $\Phi_i$ with $i \in \{ 16, 30, 31, 37, 39, 80 \}$ of \cite{Jam90}.
\end{proposition}
\begin{proof}
Suppose first that $p$ is odd. If the rank of the family is at most $4$, we have $\B_0(G) = 0$ by \cite{Bog88}. Next, if the rank equals $5$, stem groups of the family have nontrivial Bogomolov multipliers if and only if they belong to the family $\Phi_{10}$ by \cite{Hos11,Hos12,Mor12p5}. Further, if the rank is $6$, then it follows from \cite{Che13} that stem groups of the family have nontrivial Bogomolov multipliers if and only if they belong to one of the isoclinism families $\Phi_i$ with $i \in \{ 18, 20, 21, 36, 38, 39 \}$. Note that the families $\Phi_{38}$ and $\Phi_{39}$ only exist when $p > 3$. The groups in the families $\Phi_{18}$, $\Phi_{20}$ and $\Phi_{21}$ are of nilpotency class at most $3$, so none of their proper quotients and subgroups can belong to the isoclinism family $\Phi_{10}$. Hence these families are indeed $\B_0$-minimal. Central quotients of stem groups in the families $\Phi_{38}$ and $\Phi_{39}$ belong to the family $\Phi_{10}$, so these groups are not $\B_0$-minimal. On the other hand, the center of the stem groups of the family $\Phi_{36}$ is of order $p$ and the central quotients of these groups belong to the family $\Phi_9$, so this family is $\B_0$-minimal.

Now let $p = 2$. It is shown in \cite{Chu08,Chu09} that the groups of minimal order having nontrivial Bogomolov multipliers are exactly the groups forming the stem of the isoclinism family $\Gamma_{16}$ of \cite{Hal64}, so this family is $\B_0$-minimal. In the notation of \cite{Jam90}, it corresponds to $\Phi_{16}$. Now consider the isoclinism families of rank $7$. Their Bogomolov multipliers have been determined in \cite{ARXIVREF}. The families whose multipliers are nontrivial are precisely the families $\Phi_i$ with $i \in \{ 30, 31, 37, 39, 43, 58, 60, 80, 106, 114 \}$. It remains to filter out the $\B_0$-minimal families from this list. Making use of the presentations of representative groups of these families as given in \cite{ARXIVREF}, it is straightforward that stem groups of the families $\Phi_{43}$, $\Phi_{106}$ and $\Phi_{114}$ contain a maximal subgroup belonging to the family $\Phi_{16}$, which implies that these families are not $\B_0$-minimal. Similarly, stem groups of the families $\Phi_{58}$ and $\Phi_{60}$ possess maximal quotient groups belonging to $\Phi_{16}$, so these families are also not $\B_0$-minimal. On the other hand, it is readily verified that stem groups of the families $\Phi_i$ with $i \in \{ 30, 31, 37, 39, 80 \}$ have no maximal subgroups or quotients belonging to the family $\Phi_{16}$, implying that these families are $\B_0$-minimal.
\end{proof}

We now turn our attention to the structure of general $\B_0$-minimal groups. The upcoming lemma is of key importance in our approach, and will be used repeatedly throughout the paper.

\begin{lemma}
\label{l:generators}
Let $G$ be a $\B_0$-minimal $p$-group and $z = \prod_{i \in I} [x_i, y_i]$ a central element of order $p$ in $G$. Then there exist elements $a,b \in G$ satisfying 
\[	
\begin{aligned}
\textstyle
G = \langle a,b,x_i,y_i \; ; \; i \in I \rangle, \quad
[a,b] = z, \quad
a \curlywedge b \neq \prod_{i \in I} (x_i \curlywedge y_i).
\end{aligned}
\]
\end{lemma}
\begin{proof}
Let $w$ be a nontrivial element of $\B_0(G)$ and put $N = \langle z \rangle$. The canonical projection $G \to G/N$ induces a homomorphism  $G \curlywedge G \to G/N \curlywedge G/N \cong (G \curlywedge G)/J$, where $J = \langle a \curlywedge b \mid  [a,b] \in N \rangle$ by \cite[Proposition 4.1]{Mor11}. By $\B_0$-minimality of $G$, the element $w$ is in the kernel of this homomorphism, so it must belong to $J$. Suppose first that $J$ is cyclic. Then there exist elements $ x,  y \in G$ with $[x, y] = z$ and $J = \langle  x \curlywedge  y \rangle$. Since $w \in J$, we have $w = ( x \curlywedge  y)^n$ for some integer $n$. Applying the commutator mapping, we obtain $1 = [x, y]^n = z^n$, so $n$ must be divisible by $p$. But then $w = (x \curlywedge y)^n =  x^n \curlywedge  y = 1$, since $z$ is central in $G$. This shows that $J$ cannot be cyclic. Hence there exist elements $\tilde a,b \in G$ with $\prod_{i \in I} (x_i \curlywedge y_i) \notin \langle \tilde a \curlywedge b  \rangle$ and $1 \neq [\tilde a,b] \in N$. The latter implies $[\tilde a,b] = z^m$ for some integer $m$ coprime to $p$. Let $\mu$ be the multiplicative inverse of $m$ modulo $p$ and put $a = {\tilde a}^\mu$. The product $\prod_{i \in I}(x_i \curlywedge y_i) (a \curlywedge b)^{-1}$ is then a nontrivial element of $\B_0(G)$, since $[a,b]= z^{m\mu} = z$. By $\B_0$-minimality of $G$, the subgroup generated by $a, b, x_i, y_i$, $i \in I$, must equal the whole of $G$. 
\end{proof}

The above proof immediately implies the following result which can be compared with \cite[Theorem 4.6]{Bog88}.

\begin{corollary}
\label{c:exponentB0}
The Bogomolov multiplier of a $\B_0$-minimal group is of prime exponent.
\end{corollary}
\begin{proof}
Let $G$ be a $\B_0$-minimal $p$-group and $w$ a nontrivial element of $\B_0(G)$. For any central element $z$ in $G$ of order $p$, we have $w \in J_z = \langle a \curlywedge b \mid [a,b] \in \langle z \rangle \rangle$ by $\B_0$-minimality, thus $w^p = 1$, as required.
\end{proof}

We apply Lemma \ref{l:generators} to some special central elements of prime order in a $\B_0$-minimal group. In this way, some severe restrictions on the structure of $\B_0$-minimal groups are obtained. Recall that the {\it Frattini rank} of a group $G$ is the cardinality of the smallest generating set of $G$.

\begin{theorem}
\label{t:frattini}
A $\B_0$-minimal group has an abelian Frattini subgroup and is of Frattini rank at most $4$. Moreover, when the group is of nilpotency class at least $3$, it is of Frattini rank at most $3$.
\end{theorem}
\begin{proof}
Let $G$ be a $\B_0$-minimal group and $\Phi(G)$ its Frattini subgroup. Suppose that $[\Phi(G), \Phi(G)] \neq 1$. Since $G$ is a $p$-group, we have $[\Phi(G), \Phi(G)] \cap Z(G) \neq 1$, so there exists a central element $z$ of order $p$ in $[\Phi(G), \Phi(G)]$. Expand it as $z = \prod_i [x_i, y_i]$ with $x_i, y_i \in \Phi(G)$. By Lemma \ref{l:generators}, there exist $a,b \in G$ so that the group $G$ may be generated by the elements $a,b,x_i,y_i$, $i \in I$. Since the generators $x_i, y_i$ belong to $\Phi(G)$, they may be omitted, and so $G = \langle a,b \rangle$. As the commutator $[a,b]$ is central in $G$, we have $[G,G] = \langle [a,b] \rangle \cong \Z/p\Z$. It follows from here that the exponent of $G/Z(G)$ equals $p$, so we finally have $\Phi(G) = G^p [G,G] \leq Z(G)$, a contradiction. This shows that the Frattini subgroup of $G$ is indeed abelian. To show that the group $G$ is of Frattini rank at most $4$, pick any $x \in \gamma_{c-1}(G)$ and let $z = [x,y] \in \gamma_{c}(G)$ be an element of order $p$ in $G$. By Lemma \ref{l:generators}, there exist $a,b \in G$ so that the group $G$ may be generated by $a,b,x,y$. Hence $G$ is of Frattini rank at most $4$. When the nilpotency class of $G$ is at least $3$, we have $x \in \gamma_{c-1}(G) \leq [G,G]$, so the element $x$ is a nongenerator of $G$. This implies that $G$ is of Frattini rank at most $3$ in this case.
\end{proof}

Note that in particular, Theorem \ref{t:frattini} implies that a $\B_0$-minimal group is metabelian, as was already shown in \cite[Theorem 4.6]{Bog88}.

\begin{corollary}
\label{c:center}
The exponent of the center of a stem $\B_0$-minimal group divides $p^2$.
\end{corollary}
\begin{proof}
Let $G$ be a stem $\B_0$-minimal group. Then $Z(G) \leq [G,G]$, and it follows from \cite[Proposition 3.12]{Mor11} that $Z(G)$ may be generated by central commutators. For any $x,y \in G$ with $[x,y] \in Z(G)$, we have $[x^p, y^p] = 1$ by Theorem \ref{t:frattini}, which reduces to $[x,y]^{p^2} = 1$ as the commutator $[x,y]$ is central in $G$. This completes the proof. 
\end{proof}

Corollary \ref{c:center} does not apply when the $\B_0$-minimal group is not stem. The group given in Example \ref{e:256} is $\B_0$-minimal and its center is isomorphic to $\Z/2\Z \oplus \Z/8\Z$. The exponents of the upper central factors are, however, always bounded by $p$. This follows from the more general succeeding proposition.

\begin{proposition}
\label{p:highercenter}
Let $G$ be a $\B_0$-minimal group. Then $Z_2(G)$ centralizes $\Phi(G)$.
\end{proposition}
\begin{proof}
By Proposition \ref{p:pgroups}, the group $G$ is a $p$-group for some prime $p$. Suppose that $[Z_2(G), \Phi(G)] \neq 1$. Then there exists elements $x \in Z_2(G)$ and $y \in \Phi(G)$ with $[x,y] \neq 1$. By replacing $y$ with its proper power, we may assume that the commutator $[x,y]$ is of order $p$. Invoking Lemma \ref{l:generators}, we conclude that there exist elements $a,b \in G$ with $[x,y] = [a,b]$ and $x \curlywedge y \neq a \curlywedge b$. Hence $G = \langle a, b, x \rangle$ by $\B_0$-minimality. As $y \in \Phi(G)$, we have $y = \prod_i w_i^p$ for some elements $w_i \in G$. Since $x \in Z_2(G)$, this implies $[x,y] = \prod_i [x, w_i]^p$ and $x \curlywedge y = \prod_i (x \curlywedge w_i)^p$. Moreover, we may consider the $w_i$'s modulo $[G,G]$, since $x$ commutes with $[G,G]$. Putting $w_i = x^{\gamma_i} a^{\alpha_i} b^{\beta_i}$ for some integers $\alpha_i, \beta_i, \gamma_i$, we have $[x, w_i] = [x, a^{\alpha_i} b^{\beta_i}]$ and similarly for the curly wedge. By collecting the factors, we obtain $[x,y] = [x, a^{p \alpha} b^{p \beta}]$ for some integers $\alpha, \beta$. Suppose first that $p$ divides $\alpha$. Then $[x, a^{p \alpha}] = [x, a^{\alpha}]^p = [x^p, a^{\alpha}] = 1$ by Theorem \ref{t:frattini}. This implies $[x,y] = [x, b^{p \beta}]$. By an analogous argument, the prime $p$ cannot divide $\beta$, since the commutator $[x,y]$ is not trivial. Let $\bar \beta$ be the multiplicative inverse of $\beta$ modulo $p$ and put $\tilde a = a^{\bar \beta}, \tilde b = b^{\beta}$. Then we have $[\tilde a, \tilde b] = [x,y] = [x, \tilde b^p] = [x^p, \tilde b]$ and similarly $\tilde a \curlywedge \tilde b = a \curlywedge b \neq x \curlywedge y = x^p \curlywedge \tilde b$. By $\B_0$-minimality, this implies $G = \langle \tilde a, \tilde b \rangle$ with the commutator $[\tilde a, \tilde b]$ being central of order $p$ in $G$. Hence the group $G$ is of nilpotency class $2$. We now have $[\tilde a^p, b] = [\tilde a, b]^p = 1$ and similarly $[\tilde b^p, a] = 1$, so the Frattini subgroup $\Phi(G)$ is contained in the center of $G$. This is a contradiction with $[x,y] \neq 1$. Hence the prime $p$ cannot divide $\alpha$, and the same argument shows that $p$ cannot divide $\beta$. Let $\bar \alpha$ be the multiplicative inverse of $\alpha$ modulo $p$. Put $\tilde a = a^\alpha, \tilde b = b^{\bar \alpha}$. This gives $[x,y] = [x, \tilde a^p, \tilde b^{p \tilde \beta}]$ for some integer $\tilde \beta$, hence we may assume that $\alpha = 1$. Now put $\tilde a = ab$. We get $[x,y] = [x, \tilde a^p b^{p (\beta - 1)}]$.  By continuing in this manner, we degrade the exponent at the generator $b$ to $\beta = 0$, reaching a final contradiction.
\end{proof}

\begin{corollary}
\label{c:exphigher}
Let $G$ be a $\B_0$-minimal group. Then $\exp Z_i(G)/Z_{i-1}(G) = p$ for all $i \geq 2$.
\end{corollary}
\begin{proof}
It is a classical result \cite[Satz III.2.13]{Hup67} that the exponent of $Z_{i+1}(G)/Z_i(G)$ divides the exponent of $Z_i(G)/Z_{i-1}(G)$ for all $i$. Thus it suffices to prove that $\exp Z_2(G)/Z(G) = p$. To this end, let $x \in Z_2(G)$. For any $y \in G$, we have $[x^p, y] = [x,y]^p = [x, y^p] = 1$ by the preceding proposition. Hence $x^p \in Z(G)$ and the proof is complete.
\end{proof}

Corollary \ref{c:center} can, however, be improved when the group is of small enough nilpotency class.

\begin{corollary}
\label{c:center2}
The center of a stem $\B_0$-minimal group of nilpotency class $2$ is of prime exponent.
\end{corollary}
\begin{proof}
Let $G$ be a stem $\B_0$-minimal $p$-group of nilpotency class $2$. We therefore have $Z(G) = [G, G]$. For any commutator $[x,y] \in G$, Proposition \ref{p:highercenter} gives $[x,y]^p = [x^p, y] = 1$, as required.
\end{proof}

Using Corollary \ref{c:center2} together with Corollary \ref{c:exphigher}, we classify all the $\B_0$-minimal isoclinism families of nilpotency class $2$. For later use, we also determine commuting probabilities of their stem group during the course of the proof. Before stating the result, a word on the proof itself which may be of independent interest. When proving that the Bogomolov multiplier of a given group is trivial, we essentially use the technique developed in \cite{Mor12p5}. Showing nontriviality of $\B_0(G)$ is usually more difficult. One can use cohomological methods, see, for example, \cite{Hos11,Hos12}. Here, we apply the concept of $\B_0$-pairings developed in \cite{Mor11}, which essentially reduces the problem to a combinatorial one. This enables us to explicitly determine the generators of $\B_0(G)$. A {\it $\B_0$-pairing} is a map $\phi :G\times G\to H$ satisfying $\phi (xy,z)=\phi(x^y,z^y)\phi (y,z)$, $\phi (x,yz)=\phi (x,z)\phi (x^z,y^z)$ for all $x,y,z\in G$, and $\phi (a,b)=1$ for all $a,b\in G$ with $[a,b]=1$. In this case, $\phi$ determines a unique homomorphism $\phi^*:G\curlywedge G\to H$ such that $\phi^*(x\wedge y)=\phi (x,y)$ for all $x,y\in G$, see \cite{Mor11}. In order to prove that a certain element $w=\prod _i(a_i\curlywedge b_i)$ of $\B_0(G)$ is nontrivial, it therefore suffices to find a group $H$ and a $\B_0$-pairing $\phi :G\times G\to H$ such that $\prod _i\phi (a_i,b_i)\neq 1$.

\begin{theorem}
\label{t:class2}
A $\B_0$-minimal isoclinism family of nilpotency class $2$ is determined by one of the following two stem $p$-groups:
\[	
\begin{aligned}
G_1 =& \left\langle \begin{array}{l|l}
	\multirow{2}{*}{$a,b,c,d$} & a^p = b^p = c^p = d^p = 1, \\
	& [a,b] = [c,d], \, [b,d] = [a,b]^{\varepsilon}[a,c]^{\omega}, \, [a,d] = 1, \, \text{class $2$}
\end{array} \right\rangle, \\[.5em]
G_2 =& \left\langle\begin{array}{l|l} \multirow{2}{*}{$a,b,c,d$} & a^p = b^p = c^p = d^p = 1, \\
	& [a,b] = [c,d], \, [a,c] = [a,d] = 1, \, \text{class $2$}
\end{array}\right\rangle,
\end{aligned}
\]
where $\varepsilon = 1$ for $p = 2$ and $\varepsilon = 0$ for odd primes $p$, and $\omega$ is a generator of the group $(\Z/p\Z)^\times$. The groups $G_1$ and $G_2$ are of order $p^7$, their Bogomolov multipliers are $\B_0(G_1) \cong \Z/p\Z \oplus \Z/p\Z$, $\B_0(G_2) \cong \Z/p\Z$, and their commuting probabilities equal $\cp(G_1) = (p^3 + 2p^2 - p - 1)/p^6$, $\cp(G_2) = (2p^2 + p - 2)/p^5$.
\end{theorem}
\begin{proof}
Following Proposition \ref{p:families} and Proposition \ref{p:pgroups}, we may restrict ourselves to studying a stem $\B_0$-minimal $p$-group $G$ of nilpotency class $2$. This immediately implies $Z(G) = [G,G]$, and it follows from Theorem \ref{t:frattini} that the group $G$ may be generated by $4$ elements, say $a,b,c,d$, satisfying $[a,b] = [c,d]$. By Corollary \ref{c:exphigher} and Corollary \ref{c:center2}, the exponents of both $[G,G]$ and $G/[G,G]$ equal to $p$. Furthermore, the derived subgroup of $G$ is of rank at most $\binom{4}{2} - 1 = 5$, and $G/[G,G]$ is of rank at most $4$. The order of the group $G$ is therefore at most $p^9$.

Proposition \ref{p:smallb0mins} shows that no $\B_0$-minimal isoclinism families of rank at most $6$ are of nilpotency class $2$. Hence $G$ is of order at least $p^7$. Together with the above reasoning, this shows that $G$ must be of Frattini rank precisely $4$. Moreover, by possibly replacing $G$ by a group isoclinic to it, we may assume without loss of generality that $a^p = b^p = c^p = d^p = 1$. The group $G$ may therefore be regarded as a quotient of the group
\[
	K = \left\langle a,b,c,d \mid a^p = b^p = c^p = d^p = 1, \, [a,b] = [c,d], \, \text{class $2$} \right\rangle,
\]
which is of order $p^9$, nilpotency class $2$, exponent $p$ when $p$ is odd, and has precisely one commutator relation.

Suppose first that $G$ is of order precisely $p^7$. When $p = 2$, we invoke Proposition \ref{p:smallb0mins} to conclude that $G$ belongs to either the family $\Phi_{30}$ or $\Phi_{31}$ due to the nilpotency class restriction. It is readily verified using the classification \cite{Jam90} that the groups $G_1$ and $G_2$ given in the statement of the theorem are stem groups of these two families, respectively. Suppose now that $p$ is odd. The $p$-groups of order $p^7$ have been classified by O'Brien and Vaughan-Lee \cite{Obr05}, the detailed notes on such groups of exponent $p$ are available at \cite{Vau01}. Following these, we see that the only stem groups of Frattini rank $4$ and nilpotency class $2$ are the groups whose corresponding Lie algebras are labeled as $(7.16)$ to $(7.20)$ in \cite{Vau01}. In the groups arising from $(7.16)$ and $(7.17)$, the nontrivial commutators in the polycyclic presentations are all different elements of the polycyclic generating sequence. It follows from \cite{Mor12p5} that these groups have trivial Bogomolov multipliers. The remaining groups, arising from the algebras $(7.18)$ to $(7.20)$, are the following:
\[	
\begin{aligned}
	G_{18} =& \left\langle \begin{array}{l|l} \multirow{2}{*}{$a,b,c,d$} & a^p = b^p = c^p = d^p = 1, \\ & [a,c] = [a,d] = 1, [a,b] = [c,d], \, \text{class $2$} \end{array} \right\rangle, \\[.5em]
	G_{19} =& \left\langle \begin{array}{l|l} \multirow{2}{*}{$a,b,c,d$} & a^p = b^p = c^p = d^p = 1, \\ & [a,d] = 1, [b,c] = [c,d], [b,d] = [a,c], \, \text{class $2$} \end{array} \right\rangle, \\[.5em]
	G_{20} =& \left\langle \begin{array}{l|l} \multirow{2}{*}{$a,b,c,d$} & a^p = b^p = c^p = d^p = 1, \\ & [a,d] = 1, [a,b] = [c,d], [b,d] = [a,c]^{\omega}, \, \text{class $2$} \end{array} \right\rangle,
\end{aligned}
\]
where $\omega$ is a generator of the multiplicative group of units of $\Z/p\Z$.

Let us first show that $\B_0(G_{19})$ is trivial. To this end, alter the presentation of $G_{19}$ by replacing $b$ with $bd$, which allows to assume $[b,c] = 1$. Now consider an element $w \in G_{19} \curlywedge G_{19}$. Since $G_{19}$ is of nilpotency class $2$, the group $G_{19} \curlywedge G_{19}$ is abelian and may be generated by the elements
$a \curlywedge b$, $a \curlywedge c$, $a \curlywedge d$, $b \curlywedge c$, $b \curlywedge d$, $c \curlywedge d$.
These are all of order dividing $p$ and we have $a \curlywedge d = b \curlywedge c = 1$. Hence $w$ can be expanded as $w = (a \curlywedge b)^{\alpha} (a \curlywedge c)^\beta (b \curlywedge d)^\gamma (c \curlywedge d)^\delta$ for some integers $\alpha, \beta, \gamma, \delta$. Note that $w$ belongs to $\B_0(G_{19})$ precisely when its image under the commutator map is trivial. In terms of exponents, this is equivalent to $\alpha = \beta + \gamma = \delta = 0$ modulo $p$. Putting $v = (a \curlywedge c)(b \curlywedge d)^{-1}$, we thus have $\B_0(G_{19}) = \langle v \rangle$. Note, however, that $[ab^{-1}cd, cd] = [a,c][b,d]^{-1} = 1$ in $G_{19}$, which gives $v = ab^{-1}cd \curlywedge cd = 1$. Hence $\B_0(G_{19})$ is trivial.

We now turn to the group $G_{18}$ and show that $\B_0(G_{18}) \cong \Z/p\Z$. As there are no groups of nilpotency class $2$ and order at most $p^6$ with a nontrivial Bogomolov multiplier, this alone will immediately imply that $G_{18}$ is a $\B_0$-minimal group. During the course of this, we also determine the orders of centralizers of elements of $G_{18}$. Just as with the group $G_{19}$, we first see that $\B_0(G_{18}) = \langle v \rangle$ with $v = (a \curlywedge b)(c \curlywedge d)^{-1}$. It now suffices to show that the element $v$ is in fact nontrivial in $G_{18} \curlywedge G_{18}$. This is done by constructing a certain $\B_0$-pairing $\phi \colon G_{18} \times G_{18} \to \Z/p\Z$. We define this pairing on tuples of elements of $G_{18}$, written in normal form. For $x = a^{\alpha_1} b^{\alpha_2} c^{\alpha_3} d^{\alpha_4} e_1$ and $y = a^{\beta_1} b^{\beta_2} c^{\beta_3} d^{\beta_4} e_2$ with $e_1, e_2 \in [G_{18}, G_{18}]$ and $0 \leq \alpha_i, \beta_i < p$, define
\[
	\phi(x,y) = \left| \begin{smallmatrix} \alpha_1 & \alpha_2 \\ \beta_1 & \beta_2 \end{smallmatrix} \right| + p\Z.
\]
Let us show that $\phi$ is indeed a $\B_0$-pairing. Start with arbitrary elements $x,y,z$ in $G_{18}$. Note that the definition of $\phi$ depends only on representatives modulo the subgroup $\langle c,d \rangle [G_{18}, G_{18}]$ of $G_{18}$. We thus have $\phi(x^y, z) = \phi(x[x,y], z) = \phi(x,z)$ and similarly $\phi(x, y^z) = \phi(x,y)$. The definition of $\phi$ is bilinear with respect to the exponent vectors of its two components, so we also have $\phi(xy,z) = \phi(x,z) + \phi(y,z)$, and the same goes for the second component. Suppose now that $[x,y] = 1$. Expanding by the basis of $[G_{18}, G_{18}]$ gives
\[
	[x,y] = [a,b]^{
	\left| \begin{smallmatrix} \alpha_1 & \alpha_2 \\ \beta_1 & \beta_2 \end{smallmatrix} \right|
	+
	\left| \begin{smallmatrix} \alpha_3 & \alpha_4 \\ \beta_3 & \beta_4 \end{smallmatrix} \right|
	} [b,c]^{
	\left| \begin{smallmatrix} \alpha_2 & \alpha_3 \\ \beta_2 & \beta_3 \end{smallmatrix} \right|
	} [b,d]^{
	\left| \begin{smallmatrix} \alpha_2 & \alpha_4 \\ \beta_2 & \beta_4 \end{smallmatrix} \right|
	}.
\]
Assume first that $\alpha_2 \neq 0$. Choosing any $\beta_2$ uniquely determines $\beta_3$ and $\beta_4$ via the exponents at the commutators $[b,c]$ and $[b,d]$, and therefore also $\beta_1$ via the exponent at $[a,b]$. Hence $|C_G(x)| = p^4$ and $\phi(x,y) = p\Z$ in this case. Suppose now that $\alpha_2 = 0$. If $\alpha_3 = \alpha_4 = \alpha_1 = 0$, then $x$ in central in $G$ and $\phi(x,y) = p\Z$. Next, if $\alpha_3 = \alpha_4 = 0$ and $\alpha_1 > 0$, then we must have $\beta_2 = 0$ by regarding the exponent at $[a,b]$, and this is the only restriction we have on $y$, hence $|C_G(x)| = p^6$ and of course $\phi(x,y) = p\Z$. Assume now that $\alpha_3 = 0$ and $\alpha_4 > 0$, in which case we must have $\beta_2 = \beta_3 = 0$. This gives $|C_G(x)| = p^5$ and $\phi(x,y) = p\Z$. Finally, when $\alpha_3 > 0$, we get $\beta_2 = 0$ from the exponent at the commutator $[b,c]$, from where it follows that $|C_G(x)| = p^5$ and $\phi(x,y) = p \Z$. We have thus shown that the mapping $\phi$ is a $\B_0$-pairing. Therefore $\phi$ determines a unique homomorphism of groups $\phi^* \colon G_{18} \curlywedge G_{18} \to \Z/p\Z$ such that $\phi^*(x \curlywedge y) = \phi(x,y)$ for all $x,y \in G_{18}$. As we have $\phi^*(v) = \phi(a,b) - \phi(c,d) = 1 + p\Z$, the element $v$ is nontrivial in $G_{18} \curlywedge G_{18}$. Hence $\B_0(G_{18}) = \langle v \rangle \cong \Z/p\Z$, and so $G_{18}$ is a $\B_0$-minimal group. By the by, we have shown $\kk(G) = 2p^4 + p^3 - 2p^2$, giving the claimed commuting probability. In the statement of the theorem, the group $G_{18}$ corresponds to $G_2$.

At last, we deal with the group $G_{20}$. During the course of this, we also determine the orders of centralizers of its elements.
It is verified as with the group above that $\B_0(G_{20}) = \langle v_1, v_2 \rangle$ with $v_1 = (c \curlywedge d)(a \curlywedge b)^{-1}$ and $v_2 = (b \curlywedge d)(a \curlywedge c)^{-\omega}$.
We now construct a $\B_0$-pairing $\phi \colon G_{20} \times G_{20} \to \Z/p\Z \times \Z/p\Z$ setting $v_1$ and $v_2$ apart. We define this pairing on tuples of elements of $G_{20}$, written in normal form. For $x = a^{\alpha_1} b^{\alpha_2} c^{\alpha_3} d^{\alpha_4} e_1$ and $y = a^{\beta_1} b^{\beta_2} c^{\beta_3} d^{\beta_4} e_2$ with $e_1, e_2 \in [G_{20}, G_{20}]$ and $0 \leq \alpha_i, \beta_i < p$, define
\[
	\phi(x,y) = \left(
	\left| \begin{smallmatrix} \alpha_2 & \alpha_1 \\ \beta_2 & \beta_1 \end{smallmatrix} \right| + p\Z
	,
	\left| \begin{smallmatrix} \alpha_2 & \alpha_4 \\ \beta_2 & \beta_4 \end{smallmatrix} \right| + p\Z
	\right).
\]
We verify as with the pairing in the previous case that $\phi$ is a bilinear mapping that depends only on representatives modulo $[G_{20}, G_{20}]$. Suppose now that $[x,y] = 1$. Expanding by the basis of $[G_{20}, G_{20}]$ gives
\[
	[x,y] = [a,b]^{ 
	\left| \begin{smallmatrix} \alpha_1 & \alpha_2 \\ \beta_1 & \beta_2 \end{smallmatrix} \right|
	+
	\left| \begin{smallmatrix} \alpha_3 & \alpha_4 \\ \beta_3 & \beta_4 \end{smallmatrix} \right|
	} [a,c]^{
	\left| \begin{smallmatrix} \alpha_1 & \alpha_3 \\ \beta_1 & \beta_3 \end{smallmatrix} \right|
	+
	\omega \left| \begin{smallmatrix} \alpha_2 & \alpha_4 \\ \beta_2 & \beta_4 \end{smallmatrix} \right|
	} [b,c]^{
	\left| \begin{smallmatrix} \alpha_2 & \alpha_3 \\ \beta_2 & \beta_3 \end{smallmatrix} \right|
	}.
\]
Assume first that $\alpha_2 \neq 0$. The vector $[\alpha_3, \beta_3]$ modulo $p$ is thus a $\Z/p\Z$-multiple of the vector $[\alpha_2, \beta_2]$ since the exponent at $[b,c]$ is divisible by $p$. Hence we have $\alpha_3 = \lambda \alpha_2, \beta_3 = \lambda \beta_2$ for some integer $\lambda$. Taking the exponents at $[a,b]$ and $[a,c]$ into account, we obtain 
\[
	\left| \begin{smallmatrix} \alpha_1 & \alpha_2 \\ \beta_1 & \beta_2 \end{smallmatrix} \right|
	+
	\lambda \left| \begin{smallmatrix} \alpha_2 & \alpha_4 \\ \beta_2 & \beta_4 \end{smallmatrix} \right|
	= 0
	\quad \text{and} \quad
	\lambda \left| \begin{smallmatrix} \alpha_1 & \alpha_2 \\ \beta_1 & \beta_2 \end{smallmatrix} \right|
	+
	\omega \left| \begin{smallmatrix} \alpha_2 & \alpha_4 \\ \beta_2 & \beta_4 \end{smallmatrix} \right| = 0.
\]
If at least one of the two determinants in the above system is nonzero modulo $p$, then the system itself has a nontrivial solution, hence its determinant is trivial. This implies $\omega = \lambda^2$ modulo $p$, which is impossible, since $\omega$ is a generator of $(\Z/p\Z)^\times$. Both of the two determinants above are therefore trivial modulo $p$, implying that choosing any $\beta_2$ uniquely determines $\beta_1, \beta_3$ and $\beta_4$. Hence $|C_G(x)| = p^4$ in this case and we also have $\phi(x,y) = (p\Z, p\Z)$. Assume now that $\alpha_2 = 0$. If $\alpha_3 > 0$, then the exponent at $[b,c]$ gives $\beta_2 = 0$, which shows that $\phi(x,y) = (p\Z,p\Z)$. Taking the exponents at $[a,b]$ and $[a,c]$ into account, we see that choosing any $\beta_3$ uniquely determines $\beta_1$ and $\beta_4$, so we have $|C_G(x)| = p^4$ in this case. Now assume that $\alpha_3 = 0$. If we also have $\alpha_4 = 0$, then the element $x$ is either central, in which case $\alpha_1 = 0$ and so $\phi(x,y)$ is trivial, or $\alpha_1 > 0$ and thus $\beta_2 = \beta_3 = 0$, which gives $|C_G(x)| = p^5$ and $\phi(x,y) = (p\Z,p\Z)$. Finally, when $\alpha_4 > 0$, we either have $\alpha_1 = 0$, which implies $\beta_2 = \beta_3 = 0$, or $\alpha_1 > 0$, in which case $\beta_2$ determines $\beta_1$ and $\beta_4$. Both cases satisfy $|C_G(x)| = p^5$ and $\phi(x,y) = (p\Z,p\Z)$. We have thus shown that the mapping $\phi$ is a $\B_0$-pairing. As such, it determines a unique homomorphism of groups $\phi^* \colon G_{20} \curlywedge G_{20} \to \Z/p\Z \times \Z/p\Z$ with the property $\phi^*(x \curlywedge y) = \phi(x,y)$ for all $x,y \in G_{20}$. Notice that $\phi^*(v_1) = \phi(c,d) - \phi(a,b) = (1 + p\Z, p\Z)$ and $\phi^*(v_2) = \phi(b,d) - \omega \phi(a,c) = (p\Z, 1 + p\Z)$, showing that the elements $v_1$ and $v_2$ are indeed nontrivial and none is contained in the subgroup generated by the other. Hence $\B_0(G_{20}) = \langle v_1, v_2 \rangle \cong \Z/p\Z \times \Z/p\Z$, and so $G_{20}$ is a $\B_0$-minimal group. By the by, we have shown $\kk(G) = p^4 + 2p^3 - p^2 - p$, giving the claimed commuting probability. In the statement of the theorem, the group $G_{20}$ corresponds to $G_1$.

So far, we have dealt with the case when the $\B_0$-minimal group $G$ is of order at most $p^7$. Were $G$ of order $p^9$, it would be isomorphic to the group $K$. By what we have shown so far, this group is not $\B_0$-minimal, since it possesses proper quotients with nontrivial Bogomolov multipliers, namely both the groups $G_1$ and $G_2$. The only remaining option is for the group $G$ to be of order $p^8$. Regarding $G$ as a quotient of $K$, this amounts to precisely one additional commutator relation being imposed in $K$, i.e., one of the commutators in $G$ may be expanded by the rest. By possibly permuting the generators, we may assume that this is the commutator $[b,d]$, so
\[
	[b,d] = [a,b]^\alpha [a,c]^\beta [a,d]^\gamma [b,c]^\delta
\]
for some integers $\alpha,\beta,\gamma,\delta$. Replacing $b$ by $ba^{-\gamma}$ and $d$ by $dc^{-\delta}$, we may further assume $\gamma = \delta = 0$.

For $p = 2$, the above expansion reduces to only $4$ possibilities. When $\alpha = \beta = 0$, interchanging $a$ with $b$ and $c$ with $d$ shows that the group $G$ possesses a proper quotient isomorphic to $G_2$. Next, when $\alpha = \beta = 0$, the group $G$ possesses a proper quotient isomorphic to $G_1$. In the case $\alpha = 1$, $\beta = 0$, replacing $c$ by $b^{-1}c$ and $a$ by $ad$ enables us to rewrite the commutator relations to $[a,b] = [c,d] = 1$. There are thus no commutator relations between the nontrivial commutators in $G$, so the Bogomolov multiplier of $G$ is trivial by \cite{Mor12p5}. Finally, when $\alpha=0$, $\beta=1$, use \cite{Jam90} to see that the group $G/\langle [a,d] \rangle$ belongs to the isoclinism family $\Phi_{31}$, thus having a nontrivial Bogomolov multiplier by Proposition \ref{p:smallb0mins}. This shows that $G$ is not a $\B_0$-minimal group in neither of these cases.

Now let $p$ be odd. The argument here is essentially the same as in the even case, only the relations need to be reduced first. For this purpose, let $\lambda$ be the multiplicative inverse of $2$ modulo $p$. Replacing $d$ by $a^{\lambda \alpha} d$, the commutator relations are rewritten to $[b,d] = [a,b]^{\lambda \alpha} [a,c]^\beta$ and $[c,d] = [a,b] [a,c]^{- \lambda \alpha}$. Put the corresponding exponents in a matrix
\[
	A = \left[ \begin{smallmatrix} \lambda \alpha & \beta \\ 1 & - \lambda \alpha \end{smallmatrix} \right].
\]
Note that the trace of $A$ equals $0$. Replacing $b$ by $\tilde b = b^{p_{11}} c^{p_{12}}$ and $c$ by $\tilde c = b^{p_{21}} c^{p_{22}}$ gives $[\tilde b,d] = [b,d]^{p_{11}} [c,d]^{p_{12}}$ and $[\tilde c, d] = [b,d]^{p_{21}} [c,d]^{p_{22}}$. Such a change of generators replaces the matrix $A$ by a similar matrix $PAP^{-1}$, where $P$ is the matrix with entries $p_{ij}$. This enables us to put $A$ in its rational canonical form. Assume first that $\det A = 0$. Since the trace of $A$ equals $0$, both eigenvalues of $A$ must be $0$. If $A$ is in fact the zero matrix, the above relations reduce to $[b,d] = [c,d] = 1$. In this case, the group $G$ possesses a proper quotient isomorphic to $G_{18}$, and so $G$ is not a $\B_0$-minimal group. The other option is that $A$ is the matrix of zeros except the $(1,2)$-entry being equal to $1$. The relations reduce to $[b,d] = [a,c]$ and $[c,d] = 1$ in this case. Interchanging $a$ with $d$, we see that $G$ again possesses a proper quotient isomorphic to $G_{18}$. Next, consider the case when $\det A$ is not a quadratic residue modulo $p$. The characteristic polynomial of $A$ is thus split, giving two eigenvalues of opposite sign. By possibly replacing $a$ by $a^k$, the matrix $A$ of relations gets scaled by $k$. We may thus assume that $A$ is the diagonal matrix with entries $1$, $-1$. This gives the relations $[b,d] = [a,b]$ and $[c,d] = [a,c]^{-1}$. Replacing $d$ by $ad$, these reduce to $[b,d] = 1$ and $[c,d] = [a,c]^{-2}$. A further replacement of $a$ by $a^{- \lambda}$ result in $[b,d] = 1$ and $[c,d] = [a,c]$. Interchanging $a$ with $b$ now gives $[a,d] = 1$ and $[b,c] = [c,d]$, and now replacing $b$ by $bd$ finally gives $[a,d] = [b,c] = 1$. There are thus no commutator relations between the nontrivial commutators in $G$, so $\B_0(G)$ must be trivial. Finally, consider the case when $\det A$ is a square residue modulo $p$. Hence $- \det A$ is a nonresidue modulo $p$, so by possibly replacing $a$ by its proper power as above, we may assume $\det A = -\omega$, where $\omega$ is the generator of the group of units of $\Z/p\Z$. The characteristic polynomial of $A$ does not split and has a trivial linear term, so the matrix $A$ equals $\left[ \begin{smallmatrix} 0 & \omega \\ 1 & 0 \end{smallmatrix} \right]$ in its rational canonical form. This gives the commutator relations $[c,d] = [a,b]$ and $[b,d] = [a,c]^{\omega}$, showing that the group $G$ possesses a proper quotient isomorphic to $G_{20}$, and is therefore not $\B_0$-minimal. The proof is complete.
\end{proof}

Note that Theorem \ref{t:class2} shows, in particular, that there exist $\B_0$-minimal groups with noncyclic Bogomolov multipliers.
We also record a corollary following from the proof of Theorem \ref{t:class2} here.

\begin{corollary}
\label{c:class2p7}
Let $G$ be a $p$-group of order $p^7$ and nilpotency class $2$. Then $\B_0(G)$ is nontrivial if and only if $G$ belongs to one of the two isoclinism families given by Theorem \ref{t:class2}. Moreover, the stem groups of these families are precisely the groups of minimal order that have nontrivial Bogomolov multipliers and are of nilpotency class $2$.
\end{corollary}

In general, there is no upper bound on the nilpotency class of a stem $\B_0$-minimal group. We show this by means of constructing a stem $\B_0$-minimal $2$-group of order $2^n$ and nilpotency class $n-3$ for any $n \geq 6$. Note that since the nilpotency class is an isoclinism invariant, this gives an infinite number of isoclinism families whose Bogomolov multipliers are all nontrivial. As we use Corollary \ref{c:cp} to do this, the example is provided in Section \ref{s:app}. On the other hand, the bound on the exponent of the center provided by Corollary \ref{c:center} together with the bound on the number of generators given by Theorem \ref{t:frattini} show that fixing the nilpotency class restricts the number of $\B_0$-minimal isoclinism families.

\begin{corollary}
\label{c:nilpotency}
Given a prime $p$ and nonegative integer $c$, there are only finitely many $\B_0$-minimal isoclinism families containing a $p$-group of nilpotency class $c$.
\end{corollary}
\begin{proof}
The exponent of a $\B_0$-minimal $p$-group of class at most $c$ is bounded above by $p^{c+1}$ using Corollary \ref{c:center} and Corollary \ref{c:exphigher}. Since $\B_0$-minimal groups may be generated by at most $4$ elements by Theorem \ref{t:frattini}, each one is an epimorphic image of the $c$-nilpotent quotient of the free 4-generator Burnside group $B(4, p^{c+1})$ of exponent $p^{c+1}$, which is a finite group. As a $\B_0$-minimal isoclinism family is determined by its stem groups, the result follows.
\end{proof}

Lastly, we say something about the fact that the Frattini subgroup of a $\B_0$-minimal group $G$ is abelian. The centralizer $C = C_G(\Phi(G))$ is of particular interest, as a classical result of Thompson, cf. \cite{Fei63}, states that $C$ is a critical group. The elements of $G$ whose centralizer is a maximal subgroup of $G$ are certainly contained in $C$. These elements have been studied by Mann in \cite{Man06}, where they are termed to have {\em minimal breadth}. We follow Mann in denoting by ${\mathcal M}(G)$ the subgroup of $G$ generated by the elements of minimal breadth. Later on, we will be dealing separately with $2$-groups. It is shown in \cite[Theorem 5]{Man06} that in this case, the nilpotency class of ${\mathcal M}(G)$ does not exceed $2$, and that the group ${\mathcal M}(G)/Z(G)$ is abelian. We show that for $\B_0$-minimal groups, the group ${\mathcal M}(G)$ is actually abelian.

\begin{proposition}
\label{p:minimalbreadth}
Let $G$ be a $\B_0$-minimal $2$-group. Then ${\mathcal M}(G)$ is abelian.
\end{proposition}
\begin{proof}
Suppose that there exist elements $g,h \in G$ of minimal breadth with $[g,h] \neq 1$. Since the group ${\mathcal M}(G)/Z(G)$ is abelian, we have $[g,h] \in Z(G)$. Without loss of generality, we may assume that $[g,h]$ is of order $2$, otherwise replace $g$ by its power. Putting $z = [g,h]$ and applying Lemma \ref{l:generators}, there exist elements $a,b \in G$ such that $G = \langle g,h,a,b \rangle$ and $a \curlywedge b \neq g \curlywedge h$ and $[a,b] = z$. Suppose that $[g,a] \neq 1$ and $[g,b] \neq 1$. Since $g$ is of minimal breadth, we have $[g,a] = [g,b]$ and so the element $\tilde a = b^{-1}a$ centralizes $g$. This gives $z = [a,b] = [\tilde a,b]$ and similarly $\tilde a \curlywedge b = a \curlywedge b$. By $\B_0$-minimality, it follows that $G = \langle g, h, \tilde a, b \rangle$ with $[g,\tilde a] = 1$. We may thus {\em a priori} assume that $[g,a]$ = 1. Now suppose $[g,b] \neq 1$. Since $g$ is of minimal breadth, we have $[g,b] = [g,h]$ and so the element $hb^{-1}$ centralizes $g$. This gives $g \curlywedge h = g \curlywedge (hb^{-1})b = g \curlywedge b$. The product $(g \curlywedge b)(a \curlywedge b)^{-1}$ is thus a nontrivial element of $\B_0(G)$. By $\B_0$-minimality, it follows that $G = \langle g, a, b \rangle$ with $[g,a] = 1$ and $[g,b] = [a,b] \in Z(G)$. Putting $\tilde g = g a^{-1}$, we have $[\tilde g, a] = 1$ and $[\tilde g,b] = 1$ with $G = \langle \tilde g, a, b \rangle$. This implies $[G,G] = \langle [a,b] \rangle \leq Z(G)$, so $G$ is of nilpotency class $2$. Therefore $G$ belongs to one of the two families given in Theorem \ref{t:class2}. The groups in both of these families have their derived subgroups isomorphic to $(\Z/2\Z)^3$. This is a contradiction, so we must have $[g,b] = 1$. Hence $G = \langle g, h, a, b \rangle$ with $[g,a] = [g,b] = 1$ and $[g,h] = z \in Z(G)$. By the same arguments applied to $h$ instead of $g$, we also have $[h,a] = [h,b] = 1$. This implies $[G,G] = \langle [a,b] \rangle \leq Z(G)$, so $G$ is again of class $2$, giving a final contradiction.
\end{proof}

\section{Commuting probability}
\label{s:cp}

\noindent This section is devoted to proving Theorem \ref{t:cp-p}.
The restrictions on the structure of $\B_0$-minimal groups, for the most part Theorem \ref{t:frattini}, are used to reduce the claim to groups belonging to isoclinism families of smallish rank.
These special cases are dealt with in the following lemma.

\begin{lemma}
\label{l:smallcp}
Let $G$ be a finite $p$-group belonging to an isoclinism family of rank at most $6$ for odd $p$, or at most $7$ for $p=2$. If $\cp(G) > (2p^2 + p - 2)/p^5$, then $\B_0(G)$ is trivial.
\end{lemma}
\begin{proof}
Both commuting probability and the Bogomolov multiplier are isoclinism invariants, see \cite{Les95,Mor14}.
It thus suffices to verify the lemma for the isoclinism families of groups with nontrivial multipliers given in \cite{ARXIVREF,Che13}.
For odd primes, commuting probabilities of such families are given in \cite[Table 4.1]{Jam80}.
The bound $(2p^2 + p - 2)/p^5$ is attained with the families $\Phi_{10}$, $\Phi_{18}$ and $\Phi_{20}$, while the rest of them have smaller commuting probabilities.
Similarly, commuting probabilities of such families of $2$-groups of rank at most $7$ are given in \cite[Table II]{Jam90}.
The bound $1/4$ is attained with the families $\Phi_{16}$ and $\Phi_{31}$, while the rest indeed all have smaller commuting probabilities.
This proves the lemma.
\end{proof}

\begin{proof}[Proof of Theorem \ref{t:cp-p}]
Assume that $G$ is a $p$-group of the smallest possible order satisfying $\cp(G) > (2p^2 + p - 2)/p^5$ and $\B_0(G)\neq 0$.
As both commuting probability and the Bogomolov multiplier are isoclinism invariants, we can assume without loss of generality that $G$ is a stem group.
By \cite{Gur06}, commuting probability of a subgroup or a quotient of $G$ exceeds $(2p^2 + p - 2)/p^5$, so all proper subgroups and quotients of $G$ have a trivial multiplier by minimality of $G$. 
This implies that $G$ is a stem $\B_0$-minimal group.
By Lemma \ref{l:smallcp}, we may additionally assume that $G$ is of order at least $p^7$ for odd primes $p$, and $2^8$ for $p=2$.

Suppose first that $G$ is of nilpotency class $2$.
By Theorem \ref{t:class2}, $G$ belongs to one of the isoclinism families given by the two stem groups in the theorem.
The groups in both of these families have commuting probability at most $(2p^2 + p - 2)/p^5$, which is in conflict with the restriction on $\cp(G)$.
So we may assume from now on that the group $G$ is of nilpotency class at least $3$.
Hence $G$ has an abelian Frattini subgroup of index at most $p^3$ by Theorem \ref{t:frattini}.
Note also that $|G:\Phi(G)| \geq p^2$, as $G$ is not cyclic. 

We first examine the case $|G:\Phi(G)| = p^3$.
In light of Lemma \ref{l:generators}, the generators $g_1$, $g_2$, $g_3$ of $G$ may be chosen in such a way that the commutator $[g_1, f] = [g_3, g_2]$ is central and of order $p$ for some $f \in \gamma_{c-1}(G) \leq \Phi(G)$.
Put 
\[
z = \min \{  |G:C_G(g_1^k \phi)| \mid 0 < k < p, \, \phi \in \langle g_2, g_3, \Phi(G) \rangle \}	
\]
and let $(k_1, \phi_1)$ be the pair at which the minimum is attained. 
We have $\phi_1 \equiv g_2^{\alpha_2} g_3^{\alpha_3}$ modulo $\Phi(G)$ for some $0 \leq \alpha_2, \alpha_3 < p$.
After possibly replacing $g_2$ by $g_2^{\alpha_2} g_3^{\alpha_3}$, we may assume that not both $\alpha_2, \alpha_3$ are nonzero, hence $\alpha_2 = 0$ and $\alpha_3 = 1$ without loss of generality.
Replacing $g_1$ by $\tilde g_1 = g_1^{k_1} g_3$, $f$ by $\tilde f = f^{k_1}$, and $g_2$ by $\tilde g_2 = g_2^{k_1} \tilde f$, we still have $G = \langle \tilde g_1, \tilde g_2, g_3 \rangle$ and $[\tilde g_1, \tilde f] = [g_1, f]^{k_1} [g_3, \tilde f] = [g_3, g_2^{k_1}][g_3, \tilde f] = [g_3, \tilde g_2]$ since $f \in \gamma_{c-1}(G)$.
The minimum $\min \{  |G:C_G(\tilde g_1 \phi)| \mid \phi \in \langle \tilde g_2, g_3, \Phi(G) \rangle \}$ is, however, now attained at $(1, g_3^{-1} \phi_1)$ with $g_3^{-1} \phi_1 \in \Phi(G)$.
We may therefore assume that $k_1=1$ and $\phi_1 \in \Phi(G)$.
Moreover, replacing $g_1$ by $\tilde g_1 = g_1 \phi_1$, we have both $[\tilde g_1, f] = [g_1, f] = [g_3, g_2]$ and $|G:C_G(\tilde g_1)| = z$, so we may actually assume that $\phi_1 = 1$.
Next, put 
$
x = \min \{  |G:C_G(\phi)| \mid \phi \in \langle g_2, g_3, \Phi(G) \rangle \backslash \Phi(G) \}	
$
with the minimum being attained at the pair $g_2^\alpha g_3^\beta \phi_0$ with $\phi_0 \in \Phi(G)$.
Replace the generators $g_2, g_3$ by setting $\tilde g_3 = g_2^\alpha g_3^\beta$ and choosing an element $\tilde g_2$ arbitrarily as long as $\langle g_2, g_3, \Phi(G) \rangle = \langle \tilde g_2, \tilde g_3, \Phi(G) \rangle$ and $[g_3, g_2] = [g_1, f^\kappa]$ for some $\kappa$.
This enables us to assume that the minimum $\min \{  |G:C_G(\phi)| \mid \phi \in \langle g_2, g_3, \Phi(G) \rangle \backslash \Phi(G) \}$ is attained at $g_3 \phi_0$ for some $\phi_0 \in \Phi(G)$.
Lastly, put 
\[
y = \min \{  |G:C_G(g_2^k \phi)| \mid 0 < k < p, \, \phi \in \langle g_3, \Phi(G) \rangle \}	
\]
with the minimum being attained at the pair $(k_2,\phi_2)$.
Writing $\phi_2 \equiv g_3^{\alpha_3}$ modulo $\Phi(G)$ and then replacing $g_2$ by $\tilde g_2 = g_2^{k_2} g_3^{\alpha_3}$ and $f$ by $\tilde f = f^{k_2}$ yields $G = \langle g_1, \tilde g_2, g_3 \rangle$ and $[g_1, \tilde f] = [g_3, g_2]^{k_2} = [g_3, \tilde g_2]$.
We may thus {\em a priori} assume that $k_2 = 1$ and $\phi_2 \in \Phi(G)$.
Note also that
\[
x = \min \{  |G:C_G(g_3^k \phi)| \mid 0 < k < p, \, \phi \in \Phi(G) \}	
\]
and the minimum is attained at $(1, \phi_0)$.
Moreover, by the very construction of $g_3$, we have $x \leq y$.

When any of the numbers $x$, $y$, $z$ equals $p$, the centralizer of the corresponding element is a maximal subgroup of $G$, and thus contains $\Phi(G)$.
In the case $z = p$, this implies $[g_1, f] = 1$, which is impossible, so we must have $z \geq p^2$.
As a consequence, ${\mathcal M}(G) \leq \langle g_2, g_3, \Phi(G) \rangle$. When $p=2$, the group ${\mathcal M}(G)$ is abelian by Proposition \ref{p:minimalbreadth}, and so the factor group $G/{\mathcal M}(G)$ is not cyclic \cite{Bog88}. This  implies that not both $g_2$ and $g_3$ belong to ${\mathcal M}(G)$ in this case.

Assume first that $z = p^2$.
This means that $|[g_1, G]| = p^2$.
As the group $G$ is of nilpotency class at least $3$, not both the commutators $[g_2, g_1]$ and $[g_3, g_1]$ belong to $\gamma_3(G)$.
By possibly replacing $g_2$ by $g_3$ (note that by doing so, we lose the assumption $x \leq y$, but we will not be needing it in this step), we may assume that $[g_3, g_1] \notin \gamma_3(G)$.
Hence $[g_1, G] = \{ [g_1, g_3^\alpha f^\beta] \mid 0 \leq \alpha, \beta < p \}$.
It now follows that for any $g \in \gamma_2(G)$, we have $[g_1, g] \in \{ [g_1, f^{\beta}] \mid 0 \leq \beta < p \}$, because the commutator $[g_1, g]$ itself belongs to $\gamma_3(G)$.
This implies $[g_1, \gamma_2(G), G] = 1$.
If the nilpotency class of $G$ it at least $4$, we have $[\gamma_{c-1}(G), g_1] = [\gamma_{c-2}(G), G, g_1] = [\gamma_{c-2}(G), g_1, G]$ as the group $G$ is metabelian by Theorem \ref{t:frattini}.
This gives $[\gamma_{c-1}(G), g_1] \leq [g_1, \gamma_2(G), G] = 1$, a contradiction with $[g_1, f] \neq 1$.
Hence $G$ must be of nilpotency class $3$. 

Consider the case when $p = 2$ first.
As the Frattini subgroup of $G$ is abelian, we have $[g_1^2, g_3^2] = 1$, which in turn gives $[g_1^4, g_3] = [g_1, g_3]^4 [g_1, g_3, g_1]^2 = 1$, and similarly $[g_1^4, g_2] = [g_3^4, g_1] = 1$.
Hence $g_1^4$, $g_2^4$, $g_3^4$ are all central in $G$, and therefore belong to $[G,G]$ as the group $G$ is stem.
The factor group $\gamma_2(G)/\gamma_3(G)$ is generated by the commutator $[g_3, g_1]$, and we either have $[g_2, g_1] = [g_3, g_1]$ or $[g_2, g_1] \in \gamma_3(G)$, since $|[g_1, G]| = z = 4$ and thus $[g_1, G] = \{ 1, [g_1, f], [g_1, g_3], [g_1, g_3f] \}$.
Moreover, we can assume that $f = [g_3, g_1]$.
Note that $[g_3^2, g_1] \in \gamma_3(G)$, implying $[g_3, g_1]^2 \in \gamma_3(G)$ and therefore $|\gamma_2(G)/\gamma_3(G)| = 2$.
The group $\gamma_3(G)$ is generated by the commutators $[g_3, g_1, g_1]$, $[g_3, g_1, g_2]$ and $[g_3, g_1, g_3]$, all being of order at most $2$.
If $[g_2, g_1] \in \gamma_3(G)$, we have $[g_3, g_1, g_2] = [g_3, g_1]^{-1} [g_3^{g_2}, g_1^{g_2}] = 1$, and if $[g_2, g_1] \notin \gamma_3(G)$, then we have $[g_3, g_1]^{g_2} = [g_3, g_1 [g_1, g_2]] = [g_3, g_1 [g_3, g_1]] = [g_3, g_1] [g_3, g_1, g_3]$, which gives $[g_3, g_1, g_2] = [g_3, g_1, g_3]$.
Replacing $g_2$ by $\tilde g_2 = g_2 g_3$ therefore enables us to assume $[g_3, g_1, g_2] = 1$.
This shows that $\gamma_3(G)$ is of order at most $4$, and the Hall-Witt identity \cite[Satz III.1.4]{Hup67} gives $[g_2, g_1, g_3] = 1$.
Note that $[g_3^2, g_1] \in \gamma_3(G)$ gives $[g_3^2, g_1] = [g_3, g_1, g_1]^k$ for some $k \in \{ 0, 1 \}$, hence $g_3^2 [g_3, g_1]^{-k}$ is central in $G$.
Therefore $g_3^2 \in [G,G]$ as $G$ is a stem group.
The same reasoning shows that $[g_2^2 [g_3, g_1]^k, g_1] = 1$ for some $k \in \{ 0,1 \}$.
If $k = 0$, then $g_2^2$ is central in $G$ and hence belongs to $[G,G]$.
This gives $|G| \leq 2^7$, a contradiction.
Hence $k = 1$ and we have $[g_2^2, g_1] = [g_3, g_1, g_1]$.
This further implies $[g_2^2, g_1] = [g_2, g_1]^2 [g_2, g_1, g_2] = [g_3, g_1]^2 [g_3, g_1, g_2] = [g_3, g_1]^2$, hence $[g_3, g_1^2] = [g_3, g_1]^2 [g_3, g_1, g_1] = [g_3, g_1]^2 [g_2^2, g_1] = 1$.
It follows from here that $[g_2, g_1^2] = [g_2, g_1]^2 [g_2, g_1, g_1] = [g_3, g_1]^2 [g_3, g_1, g_1] = [g_3, g_1^2] = 1$, and so $g_1^2$ is central in $G$.
This finally gives $|G| \leq 2^7$, a contradiction.

Suppose now that $p$ is odd. Commutators and powers relate to give the equality $[g_3^p, g_1] = [g_3, g_1]^p [g_3, g_1, g_3]^{\binom{p}{2}} = [g_3, g_1]^p = [g_3, g_1^p]$.
Assuming $g_3^p \curlywedge g_1 \neq g_3 \curlywedge g_1^p$ and invoking $\B_0$-minimality implies $G = \langle g_1, g_3 \rangle$, a contradiction.
Hence $g_3^p \curlywedge g_1 = g_3 \curlywedge g_1^p$.
Note that $[g_3, g_1]^p$ belongs to $\gamma_3(G)$, so we must have $[g_3, g_1]^p = [g_1, f^k]$ for some $k$.
Assuming $g_3^p \curlywedge g_1 \neq g_1 \curlywedge f^k$ and invoking $\B_0$-minimality gives $G = \langle g_1, g_3^p, f^k \rangle = \langle g_1 \rangle$, which is impossible.
Hence we also have $g_3^p \curlywedge g_1 = g_1 \curlywedge f^k$.
Recall, however, that $g_1 \curlywedge f \neq g_3 \curlywedge g_2$, which gives $g_3 \curlywedge g_1^p \neq g_3^k \curlywedge g_2$ whenever $k > 0$.
Referring to $\B_0$-minimality, a contradiction is obtained, showing that $k = 0$ and hence $[g_3, g_1]^p = 1$.
An analogous argument shows that $[g_2, g_1]^p = 1$.
The elements $g_1^p$, $g_2^p$, $g_3^p$ are therefore all central in $G$, which implies that $|G/[G,G]| = p^3$ as $G$ is a stem group.
Now consider the commutator $[g_2, g_1]$.
Should it belong to $\gamma_3(G)$, we have $\gamma_2(G) = \langle [g_3, g_1], [g_3, g_1, g_1], [g_3, g_1, g_3] \rangle$, since $[g_3, g_1, g_2] = 1$ by the Hall-Witt identity.
The latter gives the bound $|G| = |G/[G,G]| \cdot |[G,G]| \leq p^6$, a contradiction.
Now assume that $[g_2, g_1]$ does not belong to $\gamma_3(G)$.
By the restriction $|[g_1, G]| = p^2$, we must have $[g_2, g_1] \equiv [g_3^k, g_1]$ modulo $\gamma_3(G)$ for some $k > 0$.
Hence $|\gamma_2(G)/\gamma_3(G)| = p$ and $\gamma_3(G) = \langle [g_3, g_1, g_1], [g_3, g_1, g_2], [g_3, g_1, g_3] \rangle$ .
As in the case when $p=2$, we now have $[g_3, g_1]^{g_2} = [g_3, g_1 [g_1, g_2]] = [g_3, g_1 [g_1, g_3^k]] = [g_3, g_1, g_3]^{-k} [g_3, g_1]$, which furthermore gives $[g_3, g_1, g_2] = [g_3, g_1, g_3]^{-k}$.
All-in-all, we obtain the bound $|\gamma_3(G)| \leq p^2$ and therefore $|G| \leq p^6$, reaching a final contradiction.

Assume now that $z \geq p^3$.
We count the number of conjugacy classes in $G$ with respect to the generating set $g_1$, $g_2$, $g_3$.
The central elements $Z(G)$ are of class size $1$, and the remaining elements of $\Phi(G)$ are of class size at least $p$.
Any other element of $G$ may be written as a product of powers of $g_1$, $g_2$, $g_3$ and an element belonging to $\Phi(G)$.
These are of class size at least $x$, $y$, $z$, depending on the first nontrivial appearance of one of the generators.
Summing up, we have
\[
	\kk(G) \leq |Z(G)| + (|\Phi(G)| - |Z(G)|)/p + ( (p-1)/x  + p(p-1)/y + p^2(p-1)/z ) |\Phi(G)|.
\]
Note that since $G$ is a $3$-generated stem group of nilpotency class at least $3$, we have $|G/Z(G)| = |G/[G,G]| \cdot |[G,G]/Z(G)| \geq p^4$.
Applying this inequality, the commuting probability bound $(2p^2 + p - 2)/p^5 < \cp(G) = \kk(G)/|G|$, and the information on the number of generators $|G:\Phi(G)| = p^3$, we obtain
\begin{equation}
\label{e:centerindex}
	(2p + 1)/p^4 < 1/p^2 x + 1/p y + 1/z.
\end{equation}

Assume first that $x \geq p^2$.
We thus also have $y \geq p^2$, and inequality \eqref{e:centerindex} gives $z < p^3$, which is impossible.
So we must have $x = p$.
In particular, the generator $g_3$ centralizes $\Phi(G)$.
We may thus replace $g_2$ by $\tilde g_2 = g_2 \phi_2$ and henceforth assume that $|[g_2, G]| = y$.
When $p = 2$, not both $g_2$ and $g_3$ belong to ${\mathcal M}(G)$, so we have $y \geq 4$ in this case.
For odd primes $p$, assuming $y = p$ makes it possible to replace $g_3$ by $g_3 \phi_3$ and hence assume $|G:C_G(g_3)| = p$.
This implies that the commutators $[g_1, g_2]$, $[g_3, g_1]$ and $[g_3, g_2]$ all belong to $\gamma_c(G)$, which restricts the nilpotency class of $G$ to at most $2$, a contradiction.
We therefore have $y = |[g_2, G]| \geq p^2$.
Inequality \eqref{e:centerindex} now gives $z < p^4$, which is only possible for $z = p^3$.
Plugging this value in \eqref{e:centerindex}, we  obtain $y < p^3$, so we must also have $y = p^2$.
This leaves us with the case $x = p$, $y = |[g_2, G]| = p^2$, and $z = |[g_1, G]| = p^3$.

The commutator $[g_3, g_2]$ is central in $G$, and we have $[g_3, g_1, g_2] = 1$ by the Hall-Witt identity.
This implies that $[g_3, g_1, g_1]^{g_2} = [g_3, g_1, g_1[g_1, g_2]] = [g_3, g_1, g_1]$, hence $[g_3, g_1, g_1, g_2] = 1$.
The same reasoning gives $[g_3, g_1, g_1, g_1, g_2] = 1$.
Note that we must have $[g_3, g_1, g_1, g_1, g_1] = 1$ since $|[g_1, G]| = p^3$.
The commutator $[g_3, g_1, g_1, g_1]$ is therefore central in $G$, and the same argument applies to $[g_2, g_1, g_1, g_1]$.
Note also that since $|[g_2, G]| = p^2$, the commutator $[g_2, g_1, g_2]$ is equal to a power of $[g_3, g_2]$, hence central in $G$.
All together, this shows that all basic commutators of length $4$ are central in $G$, which implies that $G$ is of nilpotency class at most $4$.

The restriction $|[g_2, G]| = p^2$ implies that $[g_2, g_1^p] = [g_3, g_2]^k$ for some $k$.
Assuming $g_2 \curlywedge g_1^p \neq g_3^k \curlywedge g_2$ and invoking $\B_0$-minimality gives $G = \langle g_2, g_3 \rangle$, which is impossible.
Hence $g_2 \curlywedge g_1^p = g_3^k \curlywedge g_2$.
When $k > 0$, this gives $g_2 \curlywedge g_1^p \neq g_1^k \curlywedge f$, hence $G = \langle g_2, g_1 \rangle$, a contradiction.
Therefore $k = 0$ and we conclude $[g_2, g_1^p] = 1$, so $g_1^p$ is central in $G$.
Further inspection of the group $G$ is now based on whether or not the commutator $[g_2, g_1]$ belongs to $\gamma_3(G)$.

Suppose first that $[g_2, g_1] \in \gamma_3(G)$.
When $p$ is odd, this restriction is used to obtain $[g_2^p, g_1] = [g_2, g_1]^p [g_2, g_1, g_2]^{\binom{p}{2}} = [g_2, g_1]^p = [g_2, g_1^p] = 1$, showing that the element $g_2^p$ is central in $G$.
Furthermore, we have $\gamma_2(G)/\gamma_3(G) = \langle [g_3, g_1] \rangle$, $\gamma_3(G)/\gamma_4(G) = \langle [g_3, g_1, g_1] \rangle$, and $\gamma_4(G) = \langle [g_3, g_1, g_1, g_1] \rangle$, with all of the factor group being of order $p$.
When the nilpotency class of $G$ equals $3$, we thus obtain the bound $|G| = |G/[G,G]| \cdot |\gamma_2(G)/\gamma_3(G)| \cdot |\gamma_3(G)| \leq p^6$ for odd $p$ and $|G| \leq 2^7$ for $p=2$, a contradiction.
Now let $[g_3, g_1, g_1, g_1] \neq 1$ and consider the commutator $[g_3^p, g_1]$.
Since $|[g_1, G]| = p^3$, we have $[g_1, g_3^p] = [[g_3, g_1]^k [g_3, g_1, g_1]^l, g_1]$ for some $k, l$.
This shows that $g_3^p [g_3, g_1]^{-k} [g_3, g_1, g_1]^{-l}$ is central in $G$.
Since $G$ is a stem group, we conclude that $|G/[G,G]| = p^3$ when $p$ is odd, and $|G/[G,G]| \leq 2^4$ when $p = 2$.
Applying the same bound as above gives $|G| \leq p^6$ for odd $p$ and $|G| \leq 2^7$ for $p = 2$, a contradiction.

Now assume that $[g_2, g_1] \notin \gamma_3(G)$.
Consider the commutator $[g_2, g]$ for some $g \in \gamma_2(G)$. 
Since $|[g_2, G]| = p^2$ and $[g_2, g] \in \gamma_3(G)$, we have $[g_2, g] = [g_3, g_2]^k$ for some $k$.
Assuming $g_2 \curlywedge g \neq g_3^k \curlywedge g_2$ and invoking $\B_0$-minimality gives $G = \langle g_2, g_3 \rangle$, a contradiction.
Hence $g_2 \curlywedge g = g_3^k \curlywedge g_2$, implying $g_2 \curlywedge g \neq g_1^k \curlywedge f$ whenever $k > 0$, and it follows from here by $\B_0$-minimality that $G = \langle g_2, g_1 \rangle$, another contradiction.
We therefore have $[g_2, g] = 1$, that is $[g_2, \gamma_2(G)] = 1$.
Now consider the commutator $[g_2^p, g_1]$. 
Since $[g_2^p, g_1] \equiv [g_2, g_1]^p \equiv [g_2, g_1^p] \equiv 1$ modulo $\gamma_3(G)$, we have $[g_2^p, g_1] = [g, g_1]$ for some $g \in [G,G]$.
As $G$ is a stem group, this implies that $g_2^p \in [G,G]$.
The same reasoning applied to $g_3$ shows that $g_3^p \in [G,G]$.
Hence $|G/[G,G]| = p^3$.
At the same time, the derived subgroup $[G,G]$ is generated by the commutators $[g_3, g_1], [g_2, g_1], [g_3, g_1, g_1], [g_2, g_1, g_1], [g_3, g_1, g_1, g_1], [g_2, g_1, g_1, g_1]$.
By $[g_2, \gamma_2(G)] = 1$, we have $[G,G] = [g_1, G]$ and therefore $|[G,G]| = p^3$.
All together, the bound $|G| \leq p^6$ is obtained, giving a final contradiction.

At last we deal with the case when $|G : \Phi(G)| = p^2$.
Let $g_1$ and $g_2$ be the two generators of $G$, satisfying $[g_1, f] = [g_3, g_2]$ for some $f \in \gamma_{c-1}(G)$.
As before, put $y = \min \{ |G:C_G(\phi)| \mid \phi \in \langle g_1, g_2, \Phi(G) \rangle \backslash \Phi(G) \}$.
After possible replacing the generators, we may assume
\[
y = \min \{ |G:C_G(g_2^k \phi)| \mid 0 < k < p, \, \phi \in \Phi(G) \} = |G:C_G(g_2)|.
\]
Additionally put
\[
z = \min \{ |G:C_G(g_1^k \phi)| \mid 0 < k < p, \, \phi \in \langle g_2, \Phi(G) \rangle \}
\]
with the minimum being attained at the pair $(1, 1)$ after possibly replacing $g_1$ and $g_3$ just as in the case when $|G:\Phi(G)| = p^3$.
Note that we have $y \leq z$ by construction.
When $y = p$, the subgroup $\langle g_2, \Phi(G) \rangle$ is a maximal abelian subgroup of $G$, which implies $\B_0(G) = 0$ by \cite{Bog88}, a contradiction. 
Hence $z,y \geq p^2$.

We now count the number of conjugacy classes in $G$. In doing so, we may assume $|G/Z(G)| \geq p^4$.
To see this, suppose for the sake of contradiction that $|G/Z(G)| \leq p^3$.
As the nilpotency class of $G$ is at least $3$, its central quotient $G/Z(G)$ must therefore be nonabelian of order $p^3$.
Since $G$ is a $2$-generated stem group, we thus have $|G/[G,G]| = p^2$.
Furthermore, the derived subgroup of $G$ is equal to $\langle [g_1, g_2], [g_1, g_2, g_1], [g_1, g_2, g_2] \rangle$ with $[g_1, g_2, g_1]$ and $[g_1, g_2, g_2]$ of order dividing $p$.
We thus obtain the bound $|G| = |G/\gamma_2(G)| \cdot |\gamma_2(G)/\gamma_3(G)| \cdot |\gamma_3(G)| \leq p^5$, a contradiction.
Applying the inequality $|G/Z(G)| \geq p^4$, the commuting probability bound and the information on the number of generators, the degree equation yields
\begin{equation}
\label{e:centerindex2}
	(p+1)/p^4 < 1/py + 1/z.
\end{equation}

Assuming $y \geq p^3$, we also have $z \geq p^3$, which is in conflict with inequality \eqref{e:centerindex2}.
Hence $y = p^2$, and inequality \eqref{e:centerindex2} additionally gives $z \leq p^3$. 
As in the case when $|G:\Phi(G)| = p^3$, the latter bound restricts the nilpotency class of $G$ to at most $4$.
Note that the commutator $[g_2, g_1, g_2]$ is either trivial or equals a power of $[g_3, g_2]$ since $|[g_2, G]| = p^2$.
We therefore have $\gamma_2(G)/\gamma_3(G) = \langle [g_2, g_1] \rangle$, $\gamma_3(G)/\gamma_4(G) = \langle [g_2, g_1, g_1] \rangle$, and $\gamma_4(G) = \langle [g_2, g_1, g_1, g_1] \rangle$ with all the groups being of order $p$.
Moreover, both $g_1^{p^2}$ and $g_2^{p^2}$ are central in $G$ as the Frattini subgroup is abelian.
When $p=2$, this already gives $|G| \leq 2^{4 + 3} = 2^7$, a contradiction.
Similarly, if the group $G$ is of nilpotency class $3$, we obtain $|G| \leq p^6$, another contradiction.
Now let $p$ be odd and let $G$ be of nilpotency class $4$. 
Note that we have $[g_2, g_1^p] = [g_3, g_2]^k$ for some $k$.
This in turn gives $[g_2^p, g_1] = [g_2, g_1]^p [g_2, g_1, g_2]^{\binom{p}{2}} = [g_2, g_1]^p$.
We also have $[g_2, g_1^p] = [g_2, g_1]^p [g, g_1]$ for some $g \in [G,G]$ that satisfies $[g_2, g] = 1$.
Combining the two, we obtain $[g_2^p, g_1] = [g_2, g_1^p] [g^{-1}, g_1] [g_2, g_1, g_2]^{\binom{p}{2}} = [g_1, f]^k [g^{-1},g_1] = [h, g_1]$ for some $h \in [G,G]$.
Since $f \in \gamma_3(G)$, we also have $[g_2, h] = 1$.
This shows that $g_2^p h^{-1} \in Z(G)$.
As the group $G$ is stem, we therefore have $|G/[G,G]| \leq p^3$.
Hence $|G| = |G/[G,G]| \cdot |\gamma_2(G)| \leq p^6$, giving a final contradiction.
\end{proof}	

We use Theorem \ref{t:cp-p} to obtain a global bound on commuting probability that applies to all finite groups.

\begin{proof}[Proof of Corollary \ref{c:cp}]
Let $p$ be a prime dividing the order of $G$. By \cite[Lemma 2.6]{Bog04}, the $p$-part of $\B_0(G)$ embeds into $\B_0(S)$, where $S$ is a Sylow $p$-subgroup of $G$. At the same time, we have $\cp(S) \geq \cp(G) > 1/4$ by \cite{Gur06}, which gives $\B_0(S) = 0$ by Theorem \ref{t:cp-p}. Hence $\B_0(G) = 0$.
\end{proof}

The bound given by both Theorem \ref{t:cp-p} and Corollary \ref{c:cp} is sharp, as shown by the existence of groups given in Theorem \ref{t:class2} with commuting probability equal to $(2p^2 + p - 2)/p^5$ and a nontrivial Bogomolov multiplier. We also note that no sensible inverse of neither Theorem \ref{t:cp-p} nor Corollary \ref{c:cp} holds. As an example, let $G$ be a noncommutative group with $\B_0(G) = 0$, and take $G_n$ to be the direct product of $n$ copies of $G$. It is clear that $\cp(G_n) = \cp(G)^n$, which tends to $0$ with large $n$, and $\B_0(G) = 0$ by \cite{Kan12}. So there exist groups with arbitrarily small commuting probabilities yet trivial Bogomolov multipliers.

\section{Applications}
\label{s:app}

\noindent Using Theorem \ref{t:cp-p}, a nonprobabilistic criterion for the vanishing of the Bogomolov multiplier is first established.

\begin{corollary}
\label{c:derivedp2}
Let $G$ be a finite group. If $|[G,G]|$ is cubefree, then $\B_0(G)$ is trivial.
\end{corollary}
\begin{proof}
Let $S$ a nonabelian Sylow $p$-subgroup of $G$.
By counting only the linear characters of $S$, we obtain the bound $\kk(S) > |S:[S,S]| \geq |S|/p^2$, which further gives $\cp(S) > 1/p^2 \geq (2p^2 + p - 2)/p^5$.
Theorem \ref{t:cp-p} implies $\B_0(S) = 0$.
As the $p$-part of $\B_0(G)$ embeds into $\B_0(S)$ \cite[Lemma 2.6]{Bog04}, we conclude $\B_0(G) = 0$.
\end{proof}

The restriction to third powers of primes in Corollary \ref{c:derivedp2} is best possible, as shown by the $\B_0$-minimal groups given in Theorem \ref{t:class2}, whose derived subgroups are of order $p^3$. We remark that another way of stating Corollary \ref{c:derivedp2} is by saying that the Bogomolov multiplier of a finite extension of a group of cubefree order by an abelian group is trivial.
This may be compared with \cite[Lemma 4.9]{Bog88}.

We now apply Corollary \ref{c:cp} to provide some curious examples of $\B_0$-minimal isoclinism families, determined by their stem groups. These in particular show that there is indeed no upper bound on the nilpotency class of a $\B_0$-minimal group. 

\begin{example}
\label{e:bignilpotency}
For every $n \geq 6$, consider the group
\[
	G_n = \left\langle \begin{array}{l|l}
			 \multirow{2}{*}{$a,b,c$} & a^2 = b^2 = 1, \, c^2 = [a,c], \\
			 & [c,b] = [c, {}_{n-1} a],  \, [b,a] \text{ central}, \, \text{class $n$} 
		\end{array} \right\rangle.
\]
Another way of presenting $G_n$ is by a polycyclic generating sequence $g_i$, $1 \leq i \leq n$, subject to the following relations: 
$g_1^2 = g_2^2 = 1,$ 
$g_i^2 = g_{i+1} g_{i+2}$ for $2 < i < n-2$,
$g_{n-2}^2 = g_{n-1}$,
$g_{n-1}^2 = g_n^2 = 1$,
$[g_2, g_1] = g_n$,
$[g_i, g_1] = g_{i+1}$ for $2 < i < n-1$,
$[g_{n-1}, g_1] = [g_n, g_1] = 1$,
$[g_3, g_2] = g_{n-1}$,
and all the nonspecified commutators are trivial. 
Note that the group $G_6$ is the group given in Example \ref{e:64}.
For any $n \geq 6$, the group $G_n$ is a group of order $2^n$ and of nilpotency class $n-3$, generated by $g_1, g_2, g_3$. It is readily verified that $Z(G_n) = \langle g_{n-1}, g_n \rangle \cong \Z/2\Z \times \Z/2\Z$ and $[G_n,G_n] = \langle g_4, g_n \rangle \cong \Z/2^{n-4}\Z \times \Z/2\Z$, whence $G_n$ is a stem group.
We claim that the group $G_n$ is in fact a $\B_0$-minimal group.

As we will be using Corollary \ref{c:cp}, let us first inspect the conjugacy classes of $G_n$.
It is straightforward that centralizers of noncentral elements of $\Phi(G_n)$ are all equal to the maximal subgroup $\langle g_2, g_3 \rangle \Phi(G_n)$ of $G_n$.
Furthermore, whenever the normal form of an element $g \in G_n \backslash \Phi(G)$ with respect to the above polycyclic generating sequence does not contain $g_1$, we have $C_{G_n}(g) = \langle g \rangle \Phi(G_n)$, and when the element $g$ does have $g_1$ in its normal form, we have $C_{G_n}(g) = \langle g \rangle Z(G_n)$.
Having determined the centralizers, we count the number of conjugacy classes in $G_n$.
The central elements all form orbits of size $1$.
The elements belonging to $\Phi(G) \backslash Z(G)$ all have orbits of size $2^n/2^{n-1} = 2$ and there are $2^{n-3} - 4$ of them, which gives $2^{n-4} - 2$ conjugacy classes.
Next, the elements not belonging to $\Phi(G)$ and not having $g_1$ in their normal form have orbits of size $2^n/2^{n-2} = 4$ and there are $3 \cdot 2^{n-3}$ of them, which gives $2^{n-5}$ conjugacy classes.
Finally, the elements that do have $g_1$ in their normal form each contribute one conjugacy class depending on the representative modulo $\Phi(G)$, which gives four conjugacy classes all together.
Thus $\kk(G_n) = 2^{n-4} + 3 \cdot 2^{n-5} + 6$, and hence $\cp(G_n) = 1/2^4 + 3/2^5 + 6/2^n$.

We now show that $\B_0(G_n) \cong \Z/2\Z$.
First of all, we find the generator of $\B_0(G_n)$.
The curly exterior square $G_n \curlywedge G_n$ is generated by the elements $g_3 \curlywedge g_2$ and $g_i \curlywedge g_1$ for $2 \leq i \leq n-2$.
As $G_n$ is metabelian, the group $G_n \curlywedge G_n$ is itself abelian.
Any element $w \in G_n \curlywedge G_n$ may therefore be written in the form $w = (g_3 \curlywedge g_2)^{\beta} \prod_{i=2}^{n-2} (g_i \curlywedge g_1)^{\alpha_i}$ for some integers $\beta, \alpha_i$.
Note that $w$ belongs to $\B_0(G_n)$ precisely when $[g_3, g_2]^\beta \prod_{i=2}^{n-2} [g_i, g_1]^{\alpha_i}$ is trivial.
The latter product may be written in terms of the given polycyclic generating sequence as $\prod_{i=3}^{n-2} g_{i+1}^{\alpha_i}  g_{n-1}^{\beta} g_{n}^{\alpha_2}$.
This implies $\alpha_i = 0$ for all $2 \leq i < n-2$ and $\alpha_{n-2} + \beta \equiv 0$ modulo $2$.
Note that we have $(g_{n-2} \curlywedge g_1)^2 = g_{n-2} \curlywedge g_1^2 = 1$ and similarly $(g_3 \curlywedge g_2)^2 = 1$.
Denoting $v = (g_3 \curlywedge g_2)(g_1 \curlywedge g_{n-2})^{-1}$, we thus have $\B_0(G_n) = \langle v \rangle$ with $v$ of order dividing $2$.
Let us now show that the element $v$ is in fact nontrivial in $G_n \curlywedge G_n$.
To this end, we construct a certain $\B_0$-pairing $\phi \colon G_n \times G_n \to \Z/2\Z$.
We define this pairing on tuples of elements of $G_n$, written in normal form.
For $g = \prod_{i=1}^n g_i^{a_i}$ and $h = \prod_{i=1}^n g_i^{b_i}$, put 
\[
	\phi(g,h) = {
	\left| \begin{smallmatrix}
	a_2 & b_2 \\ a_3 & b_3
	\end{smallmatrix} \right| + 2\Z.
	}
\]
We now show that $\phi$ is indeed a $\B_0$-pairing.
It is straightforward that $\phi$ is bilinear and depends only on representatives modulo $\Phi(G_n)$.
Suppose now that $[x,y] = 1$ for some $x,y \in G_n$.
If $x \in \Phi(G_n)$, then clearly $\phi(x,y) = 2\Z$. 
On the other hand, if $x \notin \Phi(G_n)$, then we must have $y \in C_{G_n}(x) \leq \langle x \rangle \Phi(G_n)$ by above, from which it follows that $\phi(x,y) = \phi(x,x) = 2\Z$.
We have thus shown that the mapping $\phi$ is a $\B_0$-pairing.
Therefore $\phi$ determines a unique homomorphism of groups $\phi^* \colon G_n \curlywedge G_n \to \Z/2\Z$ such that $\phi^*(g \curlywedge h) = \phi(g,h)$ for all $g,h \in G_n$.
As we have $\phi^*(v) = \phi(g_3, g_2) - \phi(g_1, g_{n-2}) = 1 + 2\Z$, the element $v$ is nontrivial.
Hence $\B_0(G_n) = \langle v \rangle \cong \Z/2\Z$, as required.

The above determination of centralizers also enables us to show that every subgroup of $G_n$ has commuting probability greater than $1/4$.
Note that it suffices to prove this only for maximal subgroups of $G_n$.
To this end, let $M$ be a maximal subgroup of $G_n$.
Being of index $2$ in $G_n$, $M$ contains at least one of the elements $g_3$, $g_2$, $g_2 g_3$.
If it contains two of these, then we have $M = \langle g_2, g_3 \rangle \Phi(G)$ and so $M/Z(M) = \Z/2\Z \times \Z/2\Z$.
By \cite{Gus73}, this implies $\cp(M) = 5/8$ and we are done. 
Now assume that $M$ contains exactly one of the elements $g_3$, $g_2$, $g_2 g_3$.
The centralizer of any element in $M$ not belonging to $\Phi(G)$ is, by above, of index $2^{n-1}/2^{n-2} = 2$ in $M$.
There are $2^{n-3}$ of these elements, hence contributing $2^{n-4}$ to the number of conjugacy classes in $M$.
Similarly, the elements belonging to $\Phi(G) \backslash Z(G)$ all have their centralizer of index $2^{n-1}/2^{n-2} = 2$ in $M$ and there are $2^{n-3} - 4$ of these elements, hence contributing $2^{n-4} - 2$ conjugacy classes in $M$.
This gives $\kk(M) > 4 + 2^{n-4} + (2^{n-4} - 2) > 2^{n-3}$ and therefore $\cp(M) > 1/4$.
It now follows from Corollary \ref{c:cp} that every proper subgroup of $G_n$ has a trivial Bogomolov multiplier. 

Lastly, we verify that Bogomolov multipliers of proper quotients of $G_n$ are all trivial.
To this end, let $N$ be a proper normal subgroup of $G_n$.
If $g_{n-1} \in N$, then the elements $g_2$ and $g_3$ commute in $G_n/N$.
The group $\langle g_2, g_3 \rangle \Phi(G_n)N$ is therefore a maximal abelian subgroup of $G_n/N$, and it follows from \cite{Bog88} that $\B_0(G_n/N) = 0$.
Suppose now that $g_{n-1} \notin N$.
Note that we have $G_n/N \simeq G_n/([G_n, G_n] \cap N)$ by \cite{Hal40}.
Since the Bogomolov multiplier is an isoclinism invariant, we may assume that $N$ is contained in $[G_n, G_n] = \langle g_4, g_n \rangle \cong \Z/2^{n-4}\Z \times \Z/2\Z$ .
As $g_{n-1}$ is the only element of order $2$ in $\langle g_4 \rangle$ and $g_{n-1} \notin N$, we must have either $N = \langle g_n \rangle$ or $N = \langle g_{n-1} g_n \rangle$.
Suppose first that $N = \langle g_n \rangle$ and consider the factor group $H = G_n/N$. 
Denoting $v = (g_3 \curlywedge g_2)(g_1 \curlywedge g_{n-2})$, we show as above that $\B_0(H) = \langle v \rangle$.
Note that we have $g_2 g_{n-2} \curlywedge g_1 g_3 = (g_2 \curlywedge g_3) (g_2 \curlywedge g_1) (g_{n-2} \curlywedge g_3) (g_{n-2} \curlywedge g_1) = v$ in $H$, which implies that $v$ is trivial in $H \curlywedge H$, whence $\B_0(H) = 0$.
Now consider the case when $N = \langle g_{n-1} g_n \rangle$ and put $H = G_n/N$.
Denoting $v_1 = (g_3 \curlywedge g_2)(g_1 \curlywedge g_2)$ and $v_2 = (g_{n-2} \curlywedge g_1)(g_1 \curlywedge g_2)$, we have $\B_0(H) = \langle v_1, v_2 \rangle$.
Note that $g_1 \curlywedge g_2 g_{n-2} = v_1$ and $g_2 \curlywedge g_1 g_3 = v_2$ in $H$, which implies that $v_1$ and $v_2$ are both trivial in $H \curlywedge H$, whence $\B_0(H) = 0$.
This completes the proof of the fact that $G_n$ is a $\B_0$-minimal group.
\end{example}

As stated in the introduction, these examples contradict a part of the statement of \cite[Theorem 4.6]{Bog88} and \cite[Lemma 5.4]{Bog88}.
We found that the latter has been used in proving triviality of Bogomolov multipliers of finite simple groups \cite{Kun08}.
The claim is reduced to showing $\B_0(\Out(L))$ to be trivial for all finite simple groups $L$.
Standard arguments from \cite{Bog88} are then used to further reduce this to the case when $L$ is of type $A_n(q)$ or $D_{2m + 1}(q)$.
These two cases are dealt with using the above erroneous claim. With some minor adjustments, the argument of \cite{Kun08} can be saved as follows.
First note that the linear groups $A_n(q)$ have been treated separately in \cite{Bog04}.
We remark that the argument for the exceptional case $n=2,q=9$ uses a consequence of the above statements, see also \cite{Hos11}.
The result remains valid, since the Sylow $3$-subgroup is abelian in this case. 
As for orthogonal groups $D_{2m+1}(q)$, note that the derived subgroup of $\Out(D_{2m+1}(q))$ is a subgroup of the group of outer-diagonal automorphisms, which is isomorphic to $\Z/(4,q-1)\Z$, see \cite{Wil09}.
Corollary \ref{c:derivedp2} now gives the desired result.

Lastly, we say something about groups of small orders to which Theorem \ref{t:cp-p} may be applied.
Given an odd prime $p$, it is readily verified using \cite{Jam80} that among all isoclinism families of rank at most $5$, only the family $\Phi_{10}$ has commuting probability lower or equal than $(2p^3 + p - 2)/p^5$.
Theorem \ref{t:cp-p} therefore provides a unified explanation of the known result that Bogomolov multipliers of groups of order at most $p^5$ are trivial except for the groups belonging to the family $\Phi_{10}$ \cite{Hos11,Mor12p5}.
Next, consider the groups of order $p^6$.
Out of a total of $43$ isoclinism families, $19$ of them have commuting probabilities exceeding the above bound.
This includes all groups of nilpotency class $2$.
For $2$-groups, use the classification \cite{Jam90} to see that all commuting probabilities of groups of order at most $32$ are all greater than $1/4$.
With groups of order $64$, there are $237$ groups with the same property out of a total of $267$ groups.
Again, this may be compared with known results \cite{Chu09}.
Finally, one may use Theorem \ref{t:cp-p} on the Sylow subgroups of a given group rather than using Corollary \ref{c:cp} directly, thus potentially obtaining a better bound on commuting probability that ensures triviality of the Bogomolov multiplier.


\end{document}